\documentclass[10pt]{article}
\usepackage{latexsym,amsfonts,amsthm,amsmath,amssymb,cases,authblk,multicol,chngcntr,graphicx,algorithm2e}
\usepackage[backref]{hyperref}
\usepackage{enumerate}
\usepackage{graphicx,epstopdf,tabularx,natbib}
\usepackage{color}
\newtheorem{thm}{Theorem}[section]
\newtheorem{cor}{Corollary}[section]
\newtheorem{lemma}{Lemma}[section]
\newtheorem{prob}{Problem}[section]

\newtheorem{assumption}{Assumption}[section]
\newtheorem{remark}{Remark}[section]

\newtheorem{defn}{Definition}[section]
\newtheorem{example}{Example}[section]
\newcounter{nextauthor}
\def\mathrm{\mbox}

\allowdisplaybreaks[4]

\begin{document}
\title{{\Large \bf  Nash Equilibria of   Noncooperative/Mixed  Differential Games  with Density Constraints in Infinite Dimensions}\thanks{This work was supported by the National Natural Science Foundation of China (12171339), the Scientific and Technological Research Program of Chongqing Municipal Education Commission (KJQN202400819), and the grant from Chongqing Technology and Business University (2356004).}}
\author[a]{Zhun Gou}
\author[b]{Nan-Jing Huang \thanks{Corresponding author: nanjinghuang@hotmail.com; njhuang@scu.edu.cn}}
\author[c]{Jian-Hao Kang}
\author[d]{Jen-Chih Yao}
\affil[a]{\small\it School of Mathematics and Statistics, Chongqing Technology and Business University, Chongqing 400067, P.R. China; Chongqing Key Laboratory of Statistical Intelligent Computing and Monitoring, Chongqing Technology and Business University, Chongqing, 400067, P.R. China}
\affil[b]{Department of Mathematics, Sichuan University, Chengdu, Sichuan 610064, P.R. China}
\affil[c]{School of Mathematics, Southwest Jiaotong University, Chengdu, Sichuan 610031, P.R. China}
\affil[d]{Research Center for Interneural Computing, China Medical University Hospital, China Medical University, Taichung, Taiwan 40402, P.R. China}
\date{}
\maketitle
\begin{center}
\begin{minipage}{5.6in}
\noindent{\bf Abstract.} Motivated by Cournot models, this paper proposes novel models of the noncooperative and cooperative  differential games with density constraints in infinite dimensions, where markets consist of infinite firms and demand dynamics are governed by controlled differential equations. Markets engage in noncooperative competition with each other, while firms within each market engage in noncooperative or cooperative games. The main problems are to find the noncooperative  Nash equilibrium   (NNE) of the noncooperative differential game and the mixed Nash equilibrium (MNE) of the mixed noncooperative and cooperative differential game. Moreover, fundamental relationship is established between noncooperative/mixed differential game with density constraints and  infinite-dimensional differential variational inequalities with density constraints. By variational analysis, it is proved under two   conditions  with certain symmetry that both of the two equilibrium problems can be reduced to solving systems of finite-dimensional projection equations with integral constraints by iterative computational methods.   Crucially, the two   conditions  with certain symmetry, ensuring the uniqueness of the NNE and the MNE, provide theoretical foundations for strategic decision making regarding competitive versus cooperative market behaviors.  Finally, the theoretical framework is validated through numerical simulations demonstrating the  efficacy of our results.
\end{minipage}
\end{center}

\section{Introduction}
Denote $N=\{1,2,\cdots,n\}$ and $L=\{1,2,\cdots, l\}$. For a given  Hilbert space $\mathcal{H}$, the following notations will be used throughout this paper.

    \begin{itemize}

\item     $c_1\vee c_2$: $\max\{c_1,c_2\}$, where $c_1,c_2\in \mathbb{R}$;

 \item       $c_1\wedge c_2$: $\min\{c_1,c_2\}$, where $c_1,c_2\in \mathbb{R}$;

   \item          $(c_1,\cdots,c_n)^{\top}$:   the transpose of the row vector $(c_1,\cdots,c_n)$;

   \item        $I_{S}(\cdot)$: the indicator function on the set $S$;

    \item       $H=L^2([a,b],\mathbb{R})$: the Hilbert space of all square integrable functions $g:[a,b]\rightarrow\mathbb{R}$ with $\|g\|_H=(\int_a^bg^2(x)dx)^{\frac{1}{2}}<+\infty$, where  $a<b$;

\item      $H^n$: the Hilbert space of all  vectors $\mathbf{g}=(g_1,\cdots,g_n)^{\top}$ ($g_i\in H$, $i\in N$) with $\|\mathbf{g}\|_{H^n}^2=\sum_{i\in N}\|g_i\|^2_H<+\infty$;

     \item  $H_T=L^2([s,T],H)$: the Hilbert space of all square integrable functions $g:[a,b]\rightarrow H$ with $\|g\|_{H_T}=(\int_s^T\|g\|^2_Hdt)^{\frac{1}{2}}<+\infty$, where  $a<b$
           and $0\leq s<T$;

 \item     $\langle g_1,g_2\rangle_\mathcal{H}$: the inner product on $\mathcal{H}$, where $g_1,g_2\in \mathcal{H}$;

  \item    $\mathcal{L}(\mathcal{H})$: the Banach space of all linearly bounded operator $G:\mathcal{H}\rightarrow \mathcal{H}$  with $\|G\|_{\mathcal{L}(\mathcal{H})}=\sup_{g\in \mathcal{H}}\frac{\|Gg\|_\mathcal{H}}{\|g\|_\mathcal{H}}<+\infty$;

  \item    $\mathcal{H}'$: the dual space of $\mathcal{H}$, usually $\mathcal{H}'=\mathcal{H}$;

 \item     $C([s,T],\mathcal{H})$: the Banach space of all continuous functions $g:[s,T]\rightarrow \mathcal{H}$ with 
$\|g\|_{C([s,T],\mathcal{H})}=\sup_{t\in[s,T]}\|g(t)\|_{\mathcal{H}}<+\infty$, where $0\leq s<T$;

\item   $C([s,T],\mathcal{L}({\mathcal{H}}))$: the Banach space of all continuous functions $g:[s,T]\rightarrow \mathcal{L}({H})$  with 
$\|g\|_{C([s,T],\mathcal{L}({H}))}=\sup_{t\in[s,T]}\sup_{\eta\in H,\|\eta\|_H\neq0}\frac{\|g(t)\eta\|_H}{\|\eta\|_H}<+\infty$, where $0\leq s<T$. 
\end{itemize}

Based on the foundational framework established by Cournot's original model \citep{cournot1838recherches}, Cournot competition becomes one of the most important noncooperative games and its theory has undergone substantial expansion over last two centuries to address increasingly complex oligopoly market dynamics.   A significant development lies in the parallel investigation of discrete-time \citep{acciaio2021cournot, backhoff2023dynamic, MR2917530, wang2024stability} and continuous-time \citep{MR4815584, gori2015continuous, martinez2022payoff, MR4426096} dynamic game formulations.  Meanwhile, the Cournot competition paradigm involving an infinite number of firms has emerged as an intriguing area of research with distinct analytical challenges.  As the seminal work, \citet{puu2008stability} obtained the stability  of the static Cournot Equilibrium when the number of firms increases, while subsequent studies by  \citet{chan2015bertrand},   \citet{MR3755719}, and  \citet{graber2021nonlocal} advanced mean field differential game approaches to characterize dynamic equilibria in    Cournot competitions, and obtained the existence of the equilibria.  However, the study of the dynamic Cournot competition is, in our opinion, still at its beginning stage. Motivated by \citep[Example 3.1.3]{MR2560519}, the following example provides a fresh perspective for the dynamic Cournot competition with infinite firms in differential game setting.

\begin{example}\label{75}
 Consider  a unit mass population of firms that choose production quantity  $x$ from the set $[a,b]\subseteq\mathbb{R}$ at time $t\in[0,T]$. For each $x\in [a,b]$,  denote by $u(t,x)$  the  density of the firms choosing production quantity $x$ at   time $t$, which satisfies $u(\cdot,\cdot)\in H_T$, $\int_a^bu(t,x)dx=1$ for any $t\in[0,T]$, and $u(t,x)\geq0$ for a.e.  $x\in [a,b]$ and  a.e. $t\in[0,T]$. Obviously, $u(t,\cdot)$ is a a probability density function on $[a,b]$ at any time $t$.  For given $g(\cdot)\in H$, let $(Fg)(x)=\int_a^bxg(x)dx$ ($x\in[a,b]$), then $F$ can be regarded as an operator in $\mathcal{L}(H)$. Thus, the firms' aggregate production at each time can be given by $F(u(t,\cdot))=\int_a^bxu(t,x)dx$.  Denote $X(t,\cdot)$  the demand at time $t$  and $X(0,\cdot)=\xi(\cdot)$ the initial demand, and suppose that the variation in demand is linear with respect to the total  demand  and the total aggregate production in time duration $[t,t+\triangle]$ with $\triangle>0$, i.e.,
$X(t+\triangle,\cdot)-X(t,\cdot)=\triangle[\varsigma_1X(t,\cdot)+\varsigma_0(t,\cdot)] -\varsigma_2\triangle F(u(t,\cdot)),$
where $\varsigma_1$ and $\varsigma_2$  are all given constants, and   $\varsigma_0(\cdot,\cdot)\in C([s,T],H)$ is a  given function.  Taking limit $\triangle \rightarrow0$, then   $X(t,\cdot)\in H$ satisfies  the following  dynamic system:
\begin{equation}\label{54}
dX(t,\cdot)=[\varsigma_1X(t,\cdot)-\varsigma_2F(u(t,\cdot))+\varsigma_0(t,\cdot)]dt,\quad 
X(0,\cdot)=\xi(\cdot),
\end{equation}
which can be  uniquely determined by given $u(\cdot,\cdot)$. Denote by  $x\varsigma_7(t,x)$  the  firms' production cost function of choosing production quantity $x$ at time $t$, where $\varsigma_7(t,x)$ ($\varsigma_7(\cdot,\cdot)\in C([s,T],H)$ is the unit cost for each production quantity $x$ at time $t$; and further denote by $p(t,x)$  the market price function (inverse demand function) of production quantity $x$ at time $t$, which takes the linear form with respect to $u(t,\cdot)$ and $X(t,\cdot)$ with form: $p(t,\cdot)=\varsigma_3X(t,\cdot)+\varsigma_4F(u(t,\cdot))+\varsigma_5u(t,\cdot)+\varsigma_6(t,\cdot)$,
where $\varsigma_3$, $\varsigma_4$, $\varsigma_5$  are all given constants, and $\varsigma_6(\cdot,\cdot)\in C([s,T],H)$ is a given function.  

In some situations, one should require $a>0$, $\varsigma_2>0$, $\varsigma_3\geq0$ and $\varsigma_4\leq0$. Under mild conditions, we can obtain $p(\cdot,\cdot)\in H_T$.  For given $g(\cdot)\in H$, let   $(W_1g)(x)=xg(x)$ for all $x\in[a,b]$. Then $W_1$ can be regarded as an operator in $\mathcal{L}(H)$ and the payoffs under density $u(\cdot,\cdot)$ is given by
$V(\cdot)=\int_0^T W_1[p(t,\cdot)-\varsigma_7(t,\cdot)]dt.$
Thus, $V(\cdot)\in H$  for given $u(\cdot,\cdot)$ and $V(x)$ stands for the payoffs of production quantity $x\in [a,b]$. The firms aim to find the proper density $\widehat{u}(\cdot,\cdot)$ under suitable definitions of the equilibria, which will be introduced later. Such a density $\widehat{u}(\cdot,\cdot)$ is called the optimal strategy of the firms.
\end{example}

It is noteworthy that in Example \ref{75}, the linearities of both the demand function and the price function remain a commonly employed assumption in the existing literature (see \citep{cai2019role, chan2015bertrand, MR4237615} and the rich references therein)  because of  the following two main reasons: (i) Nonlinear demand  functions and nonlinear price functions can be reasonably approximated by their linear counterparts in numerous scenarios; (ii) The inherent simplicity and well-structured nature of linear demand and linear price function are anticipated to yield elegant and analytically tractable solutions for dynamic Cournot competition.

In contrast to the existing literature, there are three distinctive features in Example \ref{75}: (i) The market comprises an infinite number of firms, and the  firms select a density   $u(t,\cdot)$ instead of production quantity $x(t)$ at time $t$; (ii) The  demand is determined by the dynamic system \eqref{54}, which is indeed a controlled infinite-dimensional   differential equation;   (iii) Both the demand  and the payoffs  must be analyzed within an infinite-dimensional setting, as  one can check that   $F\in\mathcal{L}(H) $, $W_1\in\mathcal{L}(H)$, and $X(\cdot,\cdot)\in C([s,T],H)$ under mild conditions. However, in general it is really hard or even impossible to deal with the equilibrium problems  in which the integral constraint $\int_a^bu(t,x)dx=1$ is required to hold for any $t\in[0,T]$ (see Remark \ref{80} for explanation).  Thus one should opt for the alternative,    the piecewise constant density $u(\cdot,\cdot)$, i.e., the firms are allowed to change their strategies at finite moments $T_1,\cdots,T_l$ ($s=T_0<T_1<\cdots< T_l=T$), according to which the time interval $[s,T]$ is separated into several small time intervals $(T_{m-1},T_m]$ ($m\in L$), and the density $u(t,\cdot)$ remains  independent of $t$ on each small time interval. To this end, this paper first studies the equilibrium problems on single time interval $[s,T]$ with constant density $u(\cdot,\cdot)\in H$ and then extends the  results to the  case of piecewise constant density.

Next, focusing on single time interval $[s,T]$, let us consider the generalized case  of $n$ markets and denote $\widehat{u}(\cdot)=(\widehat{u}_1(\cdot),\cdots,\widehat{u}_n(\cdot))^{\top}$  the optimal strategy of the firms.  Markets engage in noncooperative competition with each other, and  $\{\widehat{u}_i(\cdot)\}_{i=1}^n$ collectively form a Nash equilibrium, i.e.,  the strategy $\widehat{u}_i(\cdot)$ in the $i$-th market should be chosen based on the strategies $\widehat{u}_j(\cdot)$ $(j\neq i)$ from other $n-1$ markets. While noncooperative  games are formed in each market, i.e., in the $i$-th market the firms aim to maximize the payoffs $V(x)$ of each production quantity $x\in [a,b]$ by density $\widehat{u}_i(\cdot)$ in some sense. In this setting, Example \ref{75} becomes the  infinite-dimensional dynamic Cournot competition, which is indeed an infinite-dimensional noncooperative differential  game (ID-NDG).  On the other hand, in each market the firms  may cooperate   to maximize the total profit in some cases  \citep{MR4704225, MR4550673}, i.e., the firms in the $i$-th market would like to choose the proper density $\widehat{u}_i(\cdot)$  to maximize the WAP. In this setting, cooperative manner is formed in each market, and such a game is indeed an infinite-dimensional  mixed differential game (ID-MDG).  In the current paper, the NNE   and the MNE are specified as  follows, respectively.
 \begin{itemize}
 \item ID-NDG: the NNE $\widehat{u}(\cdot)$  can be defined as the one that for any $i\in N $ and a.e. $x\in[a,b]$, $\widehat{V}_i(x)\geq\int_a^b\widehat{u}(x)\widehat{V}_i(x)dx$ ($\widehat{V}_i(\cdot)$ is the  payoffs of the $i$-th markets under the density $\widehat{u}(\cdot)$) if $\widehat{u}_i(x)\neq0$, i.e., under the NNE, in the $i$-th market the  payoffs  of    production quantity $x$   is  never less  than the weighted aggregate payoffs (WAP)      $\int_a^bu_i(x)V_i(x)dx=\langle u_i(\cdot),V_i(\cdot)\rangle_H$, provided that the density of firms choosing strategy $x$ is non-zero. In this case, the  definition extends the Nash equilibrium of the population game (see \citep{como2020imitation, mertikopoulos2018riemannian, MR2560519} for details) to the infinite-dimensional setting, which considers the comparing mentality of the firms \citep{Espinosa2015Optimal, Lacker2020Many}.

\item    ID-MDG: the objective of the firms is to find an MNE $\widehat{u}(\cdot)$ that maximizes  the WAP $\int_a^bu_i(x)V_i(x)dx$ of the $i$-th each market ($\forall i\in N$). It is well-known that   MNE  is another important concept in game theory. And in this case, finding an MNE of the ID-MDG is closely related to solving an infinite-dimensional linear-quadratic optimal control problem (see \citep{MR4572454, MR4363403} and the rich references therein).
 \end{itemize}

With the help of variational analysis, it will be shown in current paper that searching for the NNE of the ID-NDG (respectively, the MNE of the ID-MDG) is equivalent to solving the system of infinite-dimensional differential variational inequality (ID-DVI), in which the optimal pair $(\widehat{X}(\cdot),\widehat{u}(\cdot))$ is governed by an infinite-dimensional differential equation  as well as an infinite-dimensional variational inequality (ID-VI) with density constraints: (i)  an integral constraint $\int_a^bu(x)dx=1$, interpreted as an infinite-dimensional constraint due to its formulation as an inner product in a Hilbert space; (ii) a value constraint $u(x)\geq0$ a.e.  $x\in[a,b]$, representing a finite-dimensional pointwise restriction. Differential variational inequalities (DVIs), first introduced by \citet{MR2375486}, have emerged as a powerful framework for modeling complex systems in many fields, such as engineering, economics, and finance. Subsequent research has extensively explored DVIs in both finite-dimensional settings \citep{MR4512275, MR3091366, MR3232620, MR4248266, MR4676612, MR4561065} and infinite-dimensional contexts \citep{MR4924761, MR3070101, MR4596385, MR3670043, MR3581523, MR4408107}. However, ID-DVIs with density constraints remain understudied, posing significant analytical and computational challenges. Such density constraints complicate the application of classical variational analysis, necessitating novel approaches to ensure well-posedness and tractability. Thus the first purpose of this paper is to ensure the solvability of ID-DVI with density constraints under mild conditions and provide iterative method for solving its solution.

One important issue of studying   the ID-NDG and the ID-MDG, after achieving the existence and uniqueness of their Nash equilibria, is the stability  of the Nash equilibria. Roughly, the NNE of the ID-NDG (respectively, the MNE of the ID-MDG) is stable if it is  insensitive for small changes of the initial value or the parameters contained in the differential equation, the payoffs  and the constraints. While stability analysis for classical Cournot equilibria has matured significantly (see \citep{MR2917530, puu2008stability, wang2024stability, MR4426096}), the stability of equilibria in infinite-dimensional games remains less explored.  As is aforementioned   that  searching for the NNE of the ID-NDG (respectively, the MNE of the ID-MDG) is equivalent to solving the system of  ID-DVI, thus the stability   of the NNE of the ID-NDG (respectively, the MNE of the ID-MDG) can be equivalently characterized through  the stability of solutions to the ID-DVI with density constraints. Existing literature on ID-DVI stability remains limited. Prior works, such as \citet{MR3070101}, examined stability for projected dynamical systems (a subclass of DVIs) in Hilbert spaces, while  \citet{MR4066425} addressed the stability of solutions to partial differential variational inequalities in Banach spaces.  However, to the best of our knowledge, the stability results for ID-DVIs with density constraints are notably absent.  This gap motivates the second purpose of this paper: establish the stability of the ID-DVIs with  density constraints, thereby directly proving   the stability  of both   MNE and  NNE.

Summarizing the above, the present paper is  devoted to finding an NNE of the ID-NDG and the MNE of the ID-MDG,  developing the method for solving ID-DVI with density constraints, and establishing the stability of the ID-DVIs with  density constraints.  The main contributions of this paper are threefold:
 \begin{itemize}
 \item From the viewpoint of models, markets comprises an infinite number of firms,  the demand in each market is described by an infinite-dimensional differential equation, and the network structure is inherently embedded in the models. These features distinct from many existing literatures.
  \item From the viewpoint of methodology, by constructing the backward operator-valued differential equation, the ID-DVI with density constraints is transformed into an ID-VI  with density constraints. Moreover, by variational analysis, the ID-VI is further transformed into a finite-dimensional projection equation with integral constraints. To the best of our knowledge, these methods are new in studying DVIs, and can be applied to solve some other type  DVIs with constraints.
 \item From the viewpoint of conclusions, the iteration methods are derived for solving the NNE of the ID-NDG and the MNE of the ID-MDG. Moreover, two conditions with certain symmetry are respectively provided for ensuring the uniqueness of the NNE and the MNE,  which may help  firms choose strategies between competition and cooperation.
\end{itemize}

The rest of this paper is structured as follows. The next section introduces some necessary preliminaries and formulates the ID-NDG and the ID-MDG on single time interval.   Section 3 obtains the unique NNE of the ID-NDG  by iteration method and discusses the stability of the NNE    under  mild conditions.   Section 4   obtains the unique MNE of the ID-MDG by iteration method  and discusses the stability of the MNE under  mild conditions. Section 5 extends the obtained results in Sections 3-4 to the  case of piecewise constant density. Section 6  concludes this paper.

\section{Preliminaries}
In the sequel, we will frequently suppress the $t$-dependence of  Banach space-valued functions and the $x$-dependence of functions $l(\cdot)$ in $H$ and $H^n$ (i.e., $l=l(\cdot))$. For convenience, we also denote the constant $c\in \mathbb{R}$ (or the vector $c\in \mathbb{R}^n$)  as the constant function in $H$ (or $H^n$).

In this section, we formulate the  ID-NDG and the ID-MDG on single time interval $[s,T]$, i.e., the density $u$ is independent of $t\in [s,T]$. Consider the controlled state equation with infinite firms in the $i$-th markets ($i\in N$):
\begin{equation}\label{1}
dX_i(t)=[A_i(t)X_i(t)+\sum_{j\in N}B_{ij}(t)u_j+f_i(t)]dt,\\
\quad X_i(s)=\xi_i\in H,
\end{equation}
 where $A_i,B_{ij}\in C([s,T],\mathcal{L}(H))$, $f_i\in C([s,T],H)$ $(\forall i,j\in N)$, $X_i(t)$ represents the demand at time $t$ in  the $i$-th market, and $u_j(x)$ stands for the density of choosing strategy $x\in [a,b]$ in   the $i$-th market.  Moreover, $u_i$  belongs to the following  admissible control set
$$
U=\{v\in H|\int_a^bv(x)dx=\langle v,1\rangle_H=1\;\mbox{and}\;v(x)\geq0\;\mbox{a.e.}\;x\in[a,b]\},
$$
 where  $e_i$ represents the $i$-th column vector of the $n\times n$-identity matrix. $u_i\in U$ means that the total proportion in the $i$-th market must be $1$. Then, $u=(u_1,\cdots,u_n)^{\top}$ belongs to the following admissible control set:
$$
 U^n=\{u\in H^n|\langle u,e_i\rangle_{H^n}=1\;\mbox{and}\;
 u_i(x)=\langle u(x),e_i\rangle_{\mathbb{R}^n}\geq0\;\mbox{a.e.}\;x\in[a,b]\;,\;\forall i\in N\},
$$
 where  $e_i$ represents the $i$-th column vector of the $n\times n$-identity matrix as well as the corresponding
 constant function in $H^n$. Clearly, $U$ and $U^n$ are nonempty, closed and convex in $H$ and $H^n$, respectively.

To ensure the solvability of corresponding equilibriua, we need the following lemma, which can be easily proved by standard arguments  (see \citep{MR3236753, MR4363403} for further discussion).
\begin{lemma}\label{8}
For any given $u\in U^n$, equation \eqref{1} admits   a unique solution $X_i\in C([s,T],H)$.
\end{lemma}
\begin{remark}
According to \citep[Proposition 1.6]{MR3236753}, we can  interchange the order of the linearly bounded operator on $H$ (or $H^n$) and the integral with respect to $t$.
\end{remark}

In the current paper, we focus on  the following two problems.
\begin{prob}\label{2}
Find an NNE $\widehat{u}=(\widehat{u}_1,\cdots,\widehat{u}_n)\in U^n$ such  that for any $i\in N$ and a.e. $x\in[a,b]$,
\begin{equation}\label{4}
\widehat{u}_i(x)>0\quad\Rightarrow\quad\widehat{V}_i(x)\geq\langle\widehat{V}_i,\widehat{u}_i\rangle_{H},
\end{equation}
where 
\begin{equation}\label{9}
\widehat{V}_i=\alpha_i+\int_s^T(E_i(t)\widehat{X}_{i}(t)+F_i(t)\widehat{u}_i)dt+G_i\widehat{X}_{i}(T)
\end{equation}
represents the optimal payoffs in the $i$-th market. Here   $\alpha_i\in H$, $G_i\in\mathcal{L}(H)$, $E_i,F_i\in C([s,T],\mathcal{L}(H))$are all given, and $\widehat{X}_{i}(t)$ is governed by:
\begin{equation}\label{57}
d\widehat{X}_{i}(t)=[A_i(t)\widehat{X}_{i}(t)+\sum_{j\in N}B_{ij}(t)\widehat{u}_j+f_i(t)]dt,\\
\quad \widehat{X}(s)=\xi_i\in H.
\end{equation}
\end{prob}
\begin{remark}
Clearly,  $\widehat{V}_i(x)$ ($x\in[a,b]$) represents the payoffs of production $x$ in the $i$-th market and one can check that $\widehat{V}_i\in H$. If $\widehat{u}_j$ ($j=1,\cdots,i-1,i+1,\cdots,n$) are all given, then we can write $\widehat{V}_i=J_i(\widehat{u}_i)$ with $J_i:H\rightarrow H$, i.e., the strategy $\widehat{u}_i$ in the $i$-th market should be chosen based on the strategies $\widehat{u}_j$ $(j\neq i)$ from other $n-1$ markets.
\end{remark}

\begin{prob}\label{62}
For $V_i=J_i(u_i)$, find an MNE $\widehat{u}=(\widehat{u}_1,\cdots,\widehat{u}_n)\in U^n$ such  that 
$\langle \widehat{u}_i,\widehat{V}_i\rangle_H=\max_{u_i\in U}\langle u_i,V_i\rangle_H$ for any $i\in N$.
\end{prob}

\begin{remark}
For Problems \ref{2} and \ref{62}, we refer to  $\widehat{V}_i=J_i(\widehat{u}_i)$ as the value function, $(\widehat{X},\widehat{u})$ as the optimal pair satisfying \eqref{57}, and $\langle \widehat{u},\widehat{V}\rangle_H$ as the optimal weighted aggregate payoffs (OWAP), where $\widehat{X}=(\widehat{X}_1,\widehat{X}_2,\cdots,\widehat{X}_n)^{\top}$.  It is noteworthy that the network structure is inherently embedded in Problem \ref{2} in two aspects: (i) The demand $X_i$ can be influenced by the  density $u_i(x)$ in all of $n$ market, where $B_{ij}$ $(i,j\in N)$ stands for the weighted coefficients; (ii) The coefficients   $A_i,B_{ij},E_i,F_i,G_i$ are allowed to be selected as the graphon operators, which take the form $W\in\mathcal{L}(H):(Wg)(x)=\int_a^bW(x,y)g(y)dy$ ($\forall  g\in H$). Such operators are able to capture the relationships among   firms and markets to some degree (see Example \ref{58}).   For further insights into network Cournot competitions, we direct the reader to \citep{bimpikis2019cournot, cai2019role, MR4755704, MR4237615}.
\end{remark}

We recall the following definitions.  A function $K:\mathcal{H}\rightarrow \mathbb{R}$ is called Fr\'{e}chet  differentiable at $\mu\in \mathcal{H}$ (see \citep{MR4363403}) if there exists a $\kappa(\mu)\in \mathcal{H}'$ such that
$$
\lim_{h\in\mathcal{H},\|h\|_{\mathcal{H}}\rightarrow 0}\frac{|K(\mu+h)-K(\mu)-\kappa(\mu)h|}{\|h\|_{\mathcal{H}}}=0.
$$
In this case, we write $\nabla K_{\mu}=\kappa(\mu)$ and call it the Fr\'{e}chet derivative of $K$ at $\mu$.  Moreover, an operator $G^*\in\mathcal{L}(\mathcal{H})$ is called the dual operator of $G\in\mathcal{L}(\mathcal{H})$  (see \citep{MR4738264}) if
$\langle Gu,v\rangle_{\mathcal{H}}=\langle u,G^*v\rangle_{\mathcal{H}}, \;\forall u,v\in \mathcal{H}.$
Especially, if $G^*=G$, then $G$ is called self-adjoint.

Based on the definition of finite-dimensional full potential game (see \citep[Subsection 3.1.2]{MR2560519}), we can define the infinite-dimensional full potential game as follows.
\begin{defn}\label{51}
A vector $V=(V_1,\cdots,V_n)^{\top}:U^n\rightarrow H^n$ is said to be a full potential game if there exists a Fr\'{e}chet  differentiable functional $K:U^n\rightarrow \mathbb{R}$ satisfying
\begin{equation}\label{44}
\nabla K_uh=\langle V,h \rangle_{H^n}, \quad \forall h\in H^n
\end{equation}
 for any given $u\in U^n$. The function $K$, which is unique up to the addition of a constant, is called the full potential function of the full potential game $V$.
\end{defn}

\section{Existence, Uniqueness and Stability of NNEs}
In this section, we focus on the existence, uniqueness and the stability of NNEs of Problem \ref{2}.
\subsection{Existence and Uniqueness of NNEs}\label{40}
In this subsection, we aim to find the NNEs of Problem \ref{2}. To begin with, we give the following equivalent characterization for  Problem \ref{2}.
\begin{prob}\label{3}
Find $\widehat{u}=(\widehat{u}_1,\cdots,\widehat{u}_n)\in U^n$ such  that
\begin{equation}\label{5}
\langle\widehat{V}_i,Z_i-\widehat{u}_i\rangle_{H}\geq 0,\quad \forall Z_i\in U^{n}, \;\forall i\in N.
\end{equation}
\end{prob}

\begin{thm}\label{11}
Problem \ref{2} is equivalent to Problem \ref{3}.
\end{thm}
\begin{proof}
\textbf{``$\Rightarrow$ part"} For any $i\in N$ and any $Z_i\in U^{n}$, it follows from \eqref{4} that
$$
\langle\widehat{V}_i,Z_i\rangle_{H}=\int_a^b\widehat{V}_i(x)Z_i(x)dx
\geq \int_a^b\langle\widehat{V}_i,\widehat{u}_i\rangle_{H}Z_i(x)dx=\langle\widehat{V}_i,\widehat{u}_i\rangle_{H},
$$
which implies
$\langle\widehat{V}_i,Z_i-\widehat{u}_i\rangle_{H}\geq 0.$

\textbf{``$\Leftarrow$ part"} Suppose that \eqref{4} fails. Then for some $\epsilon>0$, by defining
$$\Lambda^i_{\epsilon}=\{x\in[a,b]|u_i(x)>0\;\mbox{and}\;\widehat{V}_i(x)
\leq \langle\widehat{V}_i,\widehat{u}_i\rangle_{H}-\epsilon\},$$
we have $\int_a^bI_{\Lambda^i_{\epsilon}}(x)dx>0$. It follows from \eqref{5} that
\begin{equation}\label{6}
\langle\widehat{V}_i-\langle\widehat{V}_i,\widehat{u}_i\rangle_{H},Z_i\rangle_{H}
\geq0,
\quad \forall Z_i\in U, \;\forall i\in N.
\end{equation}
On the other hand, by choosing $Z_i(x)=I_{\Lambda^i_{\epsilon}}(x)(\int_a^bI_{\Lambda^i_{\epsilon}}(x)dx)^{-1}$ for $x\in[a,b]$ in \eqref{6}, one has
$\langle\widehat{V}_i-\langle\widehat{V}_i,\widehat{u}_i\rangle_{H},Z_i\rangle_{H}\leq -\epsilon<0,$
which contradicts \eqref{6} and ends the proof. 
\end{proof}
According to Theorem \ref{11}, the optimal pair $(\widehat{X},\widehat{u})$ is the solution to the following system of ID-DVI ($\forall i\in N$):
\begin{equation}\label{76}
\begin{cases}
d\widehat{X}_i(t)=[A_i(t)\widehat{X}_i(t)+\sum_{j\in N}B_{ij}(t)\widehat{u}_j+f_i(t)]dt,\quad \widehat{X}_i(s)=\xi_i\in H,\\
\langle\widehat{V}_i,Z_i-\widehat{u}_i\rangle_{H}\geq 0,\quad \forall Z_i\in U.
\end{cases}
\end{equation}
Differing from traditional DVIs, the terminal state $\widehat{X}_i(T)$ is involved   in \eqref{76}. To solve   ID-DVI \eqref{76},  let us consider the backward operator-valued differential equation
\begin{equation}\label{7}
dY_i(t)=-(E_i(t)+Y_i(t)A_i(t))dt,\quad Y_i(T)=G_i
\end{equation}
for each $i\in N$, to which the solution is formally defined as follows.
\begin{defn}\label{22}
$Y_i\in C([s,T],\mathcal{L}(H))$ is called a strong solution to \eqref{7}, if
$$
Y_i(t)\eta = G_i\eta + \int_t^T (E_i(\tau)\eta + Y_i(\tau)A_i(\tau)\eta) d\tau, \quad \forall \eta\in H,\;\forall t\in[s,T].
$$
\end{defn}
The solvability of \eqref{7} can be ensured by the following theorem.
\begin{thm}\label{23}
For any $i\in N$, the backward operator-valued differential equation  \eqref{7} admits a unique strong solution $Y_i\in C([s,T],\mathcal{L}(H))$.
\end{thm}
\begin{proof}
Without loss of generality, we assume that $\|A_i\|_{ C([s,T],\mathcal{L}(H))}\neq 0$. Define $Y_i^{(0)}=G_i$ and 
$
Y_i^{(m+1)}(\cdot)\eta=M_i[Y_i^{(m)}(\cdot)]\eta=G_i\eta+\int_{\cdot}^T(E_i(t)+Y_i^{(m)}(t)A_i(t))\eta dt
$
for given $\eta\in H$ ($T_0<T$ is fixed), then  $M_i$ can be regraded as the operator from $C([T_0,T],\mathcal{L}(H))$ to $C([T_0,T],\mathcal{L}(H))$.
For any given $\eta\in H$, letting $\Delta Y_i^{(m)}=Y_i^{(m+1)}-Y_i^{(m)}$ and $\triangle T=T-T_0$, one has
$\Delta Y_i^{(m)}(t)\eta=\int_t^T\Delta Y_i^{(m-1)}(\tau)A_i(\tau)\eta d\tau$($\forall t\in[s,T]$).
Applying the inequality $2\langle c,d\rangle_H\leq \rho \|c\|_H^2+\frac{1}{\rho}\|d\|_H^2$ $(\forall \rho>0)$ derives
\begin{align*}
 & \sup_{t\in[T_0,T]}\|\Delta Y_i^{(m)}(t)\eta\|_H^2
=2\sup_{t\in[T_0,T]}\int_t^T\langle\Delta Y_i^{(m-1)}(\tau)A_i(\tau)\eta, \Delta Y_i^{(m)}(\tau)\eta\rangle_H d\tau\\
\leq&\sup_{t\in[T_0,T]}\int_t^T[\rho\|\Delta Y_i^{(m-1)}(\tau)A_i(\tau)\eta\|_H^2+
  \frac{1}{\rho}\|\Delta Y_i^{(m)}(\tau)\eta\|_H^2]d\tau\\
\leq&\rho\|\eta\|_H^2\int_{T_0}^T\|A_i(\tau)\|^2_{\mathcal{L}(H)}\|\Delta Y_i^{(m-1)}(\tau)\|^2_{\mathcal{L}(H)}d\tau+
\frac{\|\eta\|_H^2}{\rho}\int_{T_0}^T\|\Delta Y_i^{(m)}(\tau)\|_{\mathcal{L}(H)}^2d\tau\\
\leq& \triangle T\|\eta\|_H^2[\rho\|A_i\|^2_{C([T_0,T],\mathcal{L}(H))}\|\Delta Y_i^{(m-1)}\|^2_{C([T_0,T],\mathcal{L}(H))}+\frac{1}{\rho}\|\Delta Y_i^{(m)}\|^2_{C([T_0,T],\mathcal{L}(H))}].
\end{align*}
Since the above inequality holds for all $\eta\in H$, it follows that
$$
(1-\frac{\triangle T}{\rho})\|\Delta Y_i^{(m)}\|^2_{C([T_0,T],\mathcal{L}(H))}\leq\rho \triangle T\|A_i\|^2_{C([T_0,T],\mathcal{L}(H))}\|\Delta Y_i^{(m-1)}\|^2_{C([T_0,T],\mathcal{L}(H))}.
$$
By choosing $\rho=\|A_i\|^{-1}_{C([T_0,T],\mathcal{L}(H))}$ and $\triangle T=\frac{1}{3}\|A_i\|^{-1}_{C([T_0,T],\mathcal{L}(H))}$, one has
$$
\|\Delta Y_i^{(m)}\|_{C([T_0,T],\mathcal{L}(H))}\leq \frac{1}{\sqrt{2}}\|\Delta Y_i^{(m-1)}\|_{C([T_0,T],\mathcal{L}(H))}
$$
and so $M_i$ is contractive, which implies the existence and uniqueness of the strong solution $Y_i$ to \eqref{7} on time interval $[T_0,T]$. Repeating the above arguments finite times by ending at $T_0$, the desired result can be obtained.
\end{proof}

\begin{thm}
The value function of  the $i$-th market ($i\in N$) can be given by
\begin{equation}\label{60}
\widehat{V}_i=\alpha_i+Y_i(s)\xi_i+\int_s^TY_i(t)f_i(t)dt+\int_s^TY_i(t)\sum_{j\in N}B_{ij}(t)\widehat{u}_jdt+\int_s^TF_i(t)\widehat{u}_i dt.
\end{equation}
\end{thm}
\begin{proof}
Taking the differential of $Y_i(t)\widehat{X}_{i}(t)$, one has
$$
G_i\widehat{X}_{i}(T)-Y_i(s)\xi_i
=-\int_s^TE_i(t)\widehat{X}_{i}(t)dt+\int_s^TY_i(t)(\sum_{j\in N}B_{ij}(t)\widehat{u}_j+f_i(t))dt,
$$
and the result follows from \eqref{9} immediately.
\end{proof}
For $\alpha=(\alpha_1,\cdots,\alpha_n)^{\top}$ and $\widehat{V}=(\widehat{V}_1,\cdots,\widehat{V}_n)^{\top}$, set
\begin{equation}\label{15}
\begin{cases}
\widehat{V}=P\widehat{u}+Q=\int_s^TP(t)\widehat{u}dt+Q,\quad \mathcal{P}_i(t)=Y_i(t)B_{ii}(t)+F_i(t),\\
\overline{Y}_i=\alpha_i+Y_{i}(s)\xi_{i}+\int_s^TY_{i}(t)f_{i}(t)dt,\quad Q=(\overline{Y}_1,\cdots,\overline{Y}_n)^{\top},\\
P(t)=\begin{bmatrix}
\mathcal{P}_1(t)& Y_1(t)B_{12}(t)&\dots & Y_1(t)B_{1n}(t)\\
  Y_2(t)B_{21}(t)& \mathcal{P}_2(t)&\dots & Y_2(t)B_{2n}(t)\\
\vdots &\vdots & \ddots&\vdots\\
Y_n(t)B_{n1}(t)&Y_n(t)B_{n2}(t)&\dots & \mathcal{P}_n(t)
\end{bmatrix}_{n \times n}.
\end{cases}
\end{equation}
Then, one can check $Q\in H^n$ and $P\in \mathcal{L}(H^n)$. Another equivalent characterization for  Problem \ref{2} can be specified as follows.

\begin{prob}\label{12}
Find the solution $\widehat{u}=(\widehat{u}_1,\cdots,\widehat{u}_n)\in U^n$ to the ID-VI
\begin{equation}\label{10}
\langle P\widehat{u}+Q,Z-\widehat{u}\rangle_{H^n}\geq 0,\quad \forall Z\in U^{n}.
\end{equation}
\end{prob}
\begin{remark}\label{55}
If $P\in \mathcal{L}(H^n)$ is self-adjoint, then by choosing $K:H^n\rightarrow \mathbb{R},\;K(u)=\frac{1}{2}\langle Pu,u\rangle_{H^n}+\langle Q,u\rangle_{H^n}$ ($u\in U^n$), one can check that $\nabla K_u\in (H^n)':\nabla K_ug=\langle Pu+Q,g\rangle_{H^n}$ ($\forall g\in H^n$) and $V=Pu+Q$ satisfies \eqref{44}. Thus according to Definition \ref{51}, $V$ is a full potential game. As a consequence, Problem \ref{12} is equivalent to finding $\widehat{u}\in U^n$ to minimize $K(u)$ when $P\in \mathcal{L}(H^n)$ is self-adjoint.
\end{remark}

\begin{thm}\label{59}
Problems  \ref{3} is equivalent to Problem \ref{12}.
\end{thm}
\begin{proof}
\textbf{``$\Rightarrow$ part"} It follows from \eqref{5} that
\begin{equation}\label{13}
\langle P\widehat{u}+Q,(0,\cdots,0,Z_i-\widehat{u}_i,0,\cdots,0)^{\top}\rangle_{H^n}\geq0, \quad \forall Z_i\in U, \;\forall i\in N.
\end{equation}
Summarizing the above inequality from $1$ to $n$ gives
$$
\langle P\widehat{u}+Q,Z-\widehat{u}\rangle_{H^n}\geq 0,\quad \forall Z\in U^{n}, \;\forall i\in N,
$$
where the arbitrariness of $Z\in U^{n}$ follows from the arbitrariness of $Z_i\in U$.

\textbf{``$\Leftarrow$ part"}
By setting
$Z=(\widehat{u}_1,\cdots,\widehat{u}_{i-1},Z_i,\widehat{u}_{i+1},\cdots,\widehat{u}_n)^{\top}$ ($\forall Z_i\in U$)
in \eqref{10}, we can obtain \eqref{13} and so the equivalence  follows immediately.
\end{proof}
By Theorems \ref{11} and \ref{59},  Problems  \ref{2} is equivalent to Problem \ref{12}. It is well-known that  solving the variational inequality \eqref{10} is equivalent to finding the fixed point $\widehat{u}\in U^n$ of the   equality
\begin{equation}\label{17}
\widehat{u}=\mathcal{P}_{U^n}[\widehat{u}-\epsilon_0(P\widehat{u}+Q)]=\mathcal{P}_{U^n}[(\mathbb{I}_{H^n}-\epsilon_0P)\widehat{u}-\epsilon_0Q],
\end{equation}
where $\epsilon_0$ is a given positive constant, $\mathcal{P}_{U^n}(\cdot)$ denotes the
metric projection operator onto $U^n$. We need the following assumption.
\begin{assumption}\label{16}
Assume that $P\in \mathcal{L}(H^n)$ defined in \eqref{15} is positive definite, i.e., there exists a constant $\epsilon_P>0$ such that
$\langle PZ,Z\rangle_{H^n}\geq \epsilon_{P}\|Z\|^2_{H^n}$ for all $Z\in H^{n}$.
\end{assumption}

The theorem below provides the solvability of \eqref{17} and  Problem \ref{12}.
\begin{thm}\label{29}
Under Assumption \ref{16}, there exists a unique solution $\widehat{u}\in U^n$ to \eqref{17} when $\epsilon_0\in(0,2\epsilon_P\|P\|_{\mathcal{L}(H^n)}^{-2})$.
\end{thm}
\begin{proof}
Define a mapping $M:H^n\rightarrow U^n$ by $M(u)=\mathcal{P}_{U^n}[(\mathbb{I}_{H^n}-\epsilon_0P){u}-\epsilon_0Q]$ and consider the sequence $\{u^{(m)}\}_{m=0}^{+\infty}$ with
\begin{equation}\label{81}
u^{(m+1)}=M(u^{(m)})=\mathcal{P}_{U^n}[(\mathbb{I}_{H^n}-\epsilon_0P){u^{(m)}}-\epsilon_0Q]
\end{equation}
 and $u^{(0)}(x)=(1,\cdots\,1)^{\top}$. By setting $\triangle u^{(m)}=u^{(m+1)}-u^{(m)}$ and using the non-expansive property of metric projection operators, one has
\begin{align*}
&\|\triangle u^{(m)}\|^2_{H^n}\\
=& \|\mathcal{P}_{U^n}[(\mathbb{I}_{H^n}-\epsilon_0P){u^{(m)}}-\epsilon_0Q]-\mathcal{P}_{U^n}[(\mathbb{I}_{H^n}-\epsilon_0P){u^{(m-1)}}-\epsilon_0Q]\|^2_{H^n}\\
\leq & \|(\mathbb{I}_{H^n}-\epsilon_0P){\Delta u^{(m-1)}}\|^2_{H^n}\\
\leq & \|{\Delta u^{(m-1)}}\|^2_{H^n}-2\epsilon_0\langle \Delta u^{(m-1)}, P\Delta u^{(m-1)}\rangle_{H^n}+\epsilon^2_0\langle P\Delta u^{(m-1)}, P\Delta u^{(m-1)}\rangle_{H^n}\\
\leq&  (1-2\epsilon_0\epsilon_P+\epsilon^2_0\|P\|^2_{\mathcal{L}(H^n)})\|{\Delta u^{(m-1)}}\|^2_{H^n}\|{\Delta u^{(m-1)}}\|^2_{H^n}.
\end{align*}
Observing that $U^n$ is closed and convex in $H^n$,  $M$ forms a contraction mapping and so \eqref{17} admits a unique fixed point $\widehat{u}=\lim_{m\rightarrow+\infty}u^{(m)}$, which ends the proof.
\end{proof}
In the sequel, we need the following lemma.
\begin{lemma}\label{24}
For given $\theta\in H$, the function  $g_{\theta}:\mathbb{R}\rightarrow[0,+\infty)$, defined by $g_{\theta}(\sigma)=\int_a^b(\theta(x)-\sigma)\vee 0dx$, is continuous. Moreover, equation $g_{\theta}(\sigma)=1$ admits a unique solution $\sigma\in\mathbb{R}$.
\end{lemma}

\begin{proof}
For given $\sigma_1,\sigma_2\in\mathbb{R}$ with $\sigma_1>\sigma_2$, we have
\begin{align}\label{28}
g_{\theta}(\sigma_1)-g_{\theta}(\sigma_2)&= \int_a^b[(\theta(x)-\sigma_1)\vee 0-(\theta(x)-\sigma_2)\vee 0]dx\nonumber\\
&= \int_a^bI_{\theta(x)>\sigma_1}(\sigma_2-\sigma_1)dx-\int_a^bI_{\sigma_2<\theta(x)\leq\sigma_1}(\theta(x)-\sigma_2)dx\leq0,
\end{align}
and so
\begin{align*}
|g_{\theta}(\sigma_1)-g_{\theta}(\sigma_2)|=&\int_a^bI_{\theta(x)>\sigma_1}(\sigma_1-\sigma_2)dx+\int_a^bI_{\sigma_2<\theta(x)\leq\sigma_1}(\theta(x)-\sigma_2)dx\\
\leq&\int_a^bI_{\theta(x)>\sigma_1}(\sigma_1-\sigma_2)dx+\int_a^bI_{\sigma_2<\theta(x)\leq\sigma_1}(\sigma_1-\sigma_2)dx\\
=&\int_a^bI_{\theta(x)>\sigma_2}(\sigma_1-\sigma_2)dx
\leq(b-a)(\sigma_1-\sigma_2).
\end{align*}
Thus $g_{\theta}(\sigma)$ is continuous with respect to $\sigma$. By selecting $\sigma_3>0$ large enough with $\int_a^bI_{\theta(x)>\sigma_3}dx\leq\frac{1}{4}\|\theta\|^{-2}_H$ and $\sigma_4=-\frac{2}{b-a}-\frac{1}{\sqrt{b-a}}\|\theta\|_H$, we deduce from
\begin{equation*}
\begin{cases}
 0\leq g_{\theta}(\sigma_3)\leq\int_a^bI_{\theta(x)>\sigma_3}\theta(x)dx\leq
(\int_a^bI_{\theta(x)>\sigma_3}dx)^{\frac{1}{2}}\|\theta\|_H\leq\frac{1}{2},\\
g_{\theta}(\sigma_4)\geq \int_a^b(\theta(x)-\sigma_4)dx\geq -\sqrt{b-a}\|\theta\|_H-(b-a)\sigma_4=2
\end{cases}
\end{equation*}
that  $1$ is in the  range of $g$, where we have used Cauchy-Schwarz's inequality. This indicates that there exists at least one solution $\sigma\in[\sigma_3\wedge\sigma_4,\sigma_3\vee\sigma_4]$ to $g_{\theta}(\sigma)=1$. Furthermore, we observe from \eqref{28} that if in addition $g_{\theta}(\sigma_1)=g_{\theta}(\sigma_2)=1$, then $I_{\theta(x)<\sigma_2}=1$ a.e. $x\in[a,b]$. This implies $g_{\theta}(\sigma_2)=\int_a^b(\theta(x)-\sigma_2)\vee 0dx=0$ and contradicts $g_{\theta}(\sigma_2)=1$. Thus the uniqueness of $\sigma$ follows immediately.
\end{proof}
In general, it is very hard to solve the iteration  \eqref{81} directly because of the complexity of the density constraints in $U^n$. However,  following the methods in \citep[Lemma 3.3]{yong2020stochastic}, we have the following result, which transforms the iteration  \eqref{81} into    simpler one.
\begin{thm}\label{18}
Under Assumption \ref{16}, for fixed $u^{(m)}\in U^n$ $(m=1,2,\cdots)$, the projection
$$
u^{(m+1)}=\mathcal{P}_{U^n}[(\mathbb{I}_{H^n}-\epsilon_0P){u^{(m)}}-\epsilon_0Q]
$$
 can be solved by the following system of projection equation with integral constraints:
  \begin{equation}\label{26}
   \begin{cases}
u^{(m+1)}(x)=\mathcal{P}_{\mathbb{R}^n_+}[u^{(m)}_{\lambda}(x)]\; \mbox{a.e.} \;x\in[a,b],\\
\mbox{s.t.}\quad \langle u^{(m+1)},e_i\rangle_{H^n}=1,\quad\forall i\in N,
   \end{cases}
 \end{equation}
 where  $\mathbb{R}^n_+=\{v=(v_1,\cdots,v_n)^{\top}\in  \mathbb{R}^n\bigg|v_i\geq0,\;\forall i\in N\}$, $\lambda^{(m)}=(\lambda^{(m)}_1,\cdots,\lambda^{(m)}_n)^{\top}$, and $u^{(m)}_{\lambda}= (\mathbb{I}_{H^n}-\epsilon_0P){u^{(m)}}+\epsilon_0Q+\lambda^{(m)}$.  Furthermore,  $\lambda^{(m)}$ is uniquely determined.
\end{thm}
\begin{proof}
It is easy to check that $u^{(m+1)}$ is the unique solution to the following   optimization problem
\begin{equation}\label{27}
 \min_{u^{(m+1)}\in U^n} \frac{1}{2}\|u^{(m+1)}\|^2_{H^n}-\langle (\mathbb{I}_{H^n}-\epsilon_0P){u^{(m)}}-\epsilon_0Q,u^{(m+1)}\rangle_{H^n}.
\end{equation}
By introducing   Lagrange multiplier $\lambda^{(m)}$, we know that $u^{(m+1)}\in U^n$  is  the unique solution to the   optimization problem:
  \begin{equation*}
   \begin{cases}
\min_{u^{(m+1)}\in \overline{U}^n}[ \frac{1}{2}\|u^{(m+1)}\|^2_{H^n}-\langle u^{(m)}_{\lambda},u^{(m+1)}\rangle_{H^n}-\sum_{i\in N}\lambda^{(m)}_i],\\
\mbox{s.t.}\quad \langle u^{(m+1)},e_i\rangle_{H^n}=1,\;\forall i\in N,
\end{cases}
 \end{equation*}
where $$\overline{U}^n=\{v=(v_1,\cdots,v_n)^{\top}\in H^n|\langle v(x),e_i\rangle_{\mathbb{R}^n}\geq0,\;\mbox{a.e.}\;x\in[a,b]\;,\;\forall i\in N\}.$$ 
The optimization problem mentioned above enables us to deal with the integral constraints and the value constraints separately, and is clearly equivalent to the    variational inequality:
 \begin{equation}\label{25}
   \begin{cases}
\langle u^{(m+1)}-u^{(m)}_{\lambda},Z-u^{(m+1)}\rangle_{H^n}\geq0,\quad\forall Z\in \overline{U}^n,\\
\mbox{s.t.}\quad \langle u^{(m+1)},e_i\rangle_{H^n}=1,\;\forall i\in N,
   \end{cases}
 \end{equation}
where the solution $u^{(m+1)}\in \overline{U}^n$. The existence and uniqueness of $\lambda^{(m)}$ will be discussed later. Now we show  that \eqref{25} can be transformed to the following  variational inequality with integral constraints:
  \begin{equation}\label{19}
   \begin{cases}
\langle u^{(m+1)}(x)-u^{(m)}_{\lambda}(x),Z(x)-u^{(m+1)}(x)\rangle_{\mathbb{R}^n}\geq0,\;\mbox{a.e.}\;x\in[a,b],\forall Z(x)\in \mathbb{R}^n_+,\;\\
\mbox{s.t.}\quad \langle u^{(m+1)},e_i\rangle_{H^n}=1,\;\forall i\in N,
   \end{cases}
 \end{equation}
where $u^{(m+1)},Z\in  \overline{U}^n$. Taking integral on both sides in \eqref{19}, we know that \eqref{19}$\Rightarrow$ \eqref{25}. Now we assert that the converse still holds true. Assume \eqref{19} fails and define the set
$$
\Lambda_{\epsilon}=\{x\in\mathbb{R}^n|\langle u^{(m+1)}(x)-u^{(m)}_{\lambda}(x),Z_0(x)-u^{(m+1)}(x)\rangle_{\mathbb{R}^n}\leq-\epsilon\}
$$
 for some vector $Z_0(x)\in \mathbb{R}^n_+$ ($Z_0\in \overline{U}^n$) and some constant $\epsilon>0$, then  $\int_a^b I_{\Lambda_{\epsilon}}(x)dx>0$. Choosing $Z\in \overline{U}^n$ defined by
 $Z(x)=Z_0(x)I_{\Lambda_{\epsilon}}(x)+u^{(m+1)}(x)(1-I_{\Lambda_{\epsilon}}(x))$ ($x\in[a,b]$)
in \eqref{25}, we get
$
 \langle u^{(m+1)}-u^{(m)}_{\lambda},Z-u^{(m+1)}\rangle_{H^n}
\leq\mbox{}-\epsilon\int_a^b I_{\Lambda_{\epsilon}}(x)dx<0.
$
 This contradicts \eqref{25}, and thus \eqref{19} and \eqref{25} are equivalent. Furthermore,   projection equation  \eqref{26} is clearly equivalent to  \eqref{19}.  By using the fact that
 $$
 1=\langle u^{(m+1)},e_i\rangle_{H^n}=\int_a^b(\langle[(\mathbb{I}_{H^n}-\epsilon_0P){u^{(m)}}](x)-\epsilon_0Q(x),e_i\rangle_{\mathbb{R}^n}-\lambda^{(m)}_i)\vee0dx
 $$
for any $i\in N$, and letting $\theta(x)=\langle[(\mathbb{I}_{H^n}-\epsilon_0P){u^{(m)}}](x)-\epsilon_0Q(x),e_i\rangle_{\mathbb{R}^n}$, one has $\theta\in H$. So the existence and uniqueness of $\lambda_i^{(m)}$ ($\forall i\in N$) in \eqref{26} follows from Lemma \ref{24} directly. Finally, the unique solution $u^{(m+1)}(x)$ to projection equation  \eqref{26} satisfies
\begin{align*}
 \|u^{(m+1)}\|_{H^n}^2=&\int_a^b \|\mathcal{P}_{\mathbb{R}^n_+}([(\mathbb{I}_{H^n}-\epsilon_0P){u^{(m)}}](x)-\epsilon_0Q(x)-\lambda^{(m)})\|_{\mathbb{R}^n}^2dx\\
 \leq&\int_a^b \|[(\mathbb{I}_{H^n}-\epsilon_0P){u^{(m)}}](x)-\epsilon_0Q(x)-\lambda^{(m)}\|_{\mathbb{R}^n}^2dx \\
 \leq&3\|\mathbb{I}_{H^n}-\epsilon_0P\|_{\mathcal{L}(H^n)}^2\|u^{(m)}\|_{H^n}^2+3\epsilon_0^2\|Q^{(m)}\|_{H^n}^2+3\|\lambda^{(m)}\|_{H^n}^2.
\end{align*}
and so $u^{(m+1)}\in U^n$, which completes the proof.
 \end{proof}
The theorem below provides the convergence of the sequence $\{\lambda^{(m)}\}_{m=1}^{+\infty}$.
\begin{thm}\label{49}
Under Assumption \ref{16}, the limit $\lambda=\lim_{m\rightarrow+\infty}\lambda^{(m)}$ exists and satisfies the following   projection equation with integral constraints:
  \begin{equation}\label{30}
   \begin{cases}
\widehat{u}(x)=\mathcal{P}_{\mathbb{R}^n_+}([(\mathbb{I}_{H^n}-\epsilon_0P){\widehat{u}}](x)-\epsilon_0Q(x)-\lambda),\quad \mbox{a.e.} \;x\in[a,b],\\
\mbox{s.t.}\quad \langle \widehat{u},e_i\rangle_{H^n}=1,\quad\forall i\in N,
   \end{cases}
 \end{equation}
 where $\widehat{u}$ is the unique solution to Problem \ref{12}, and is as well the unique NNE of Problem \ref{2}.
\end{thm}
 \begin{proof}
 By Lemma \ref{24}, one can derive from
  $$
 1=\langle \widehat{u},e_i\rangle_{H^n}=\int_a^b(\langle[(\mathbb{I}_{H^n}-\epsilon_0P){\widehat{u}}](x)-\epsilon_0Q(x),e_i\rangle_{\mathbb{R}^n}-\lambda_i)\vee0dx,
 \quad \forall i\in N
 $$
 that there exists a unique $\lambda_i\in \mathbb{R}$ ($\forall i\in N$) satisfying \eqref{30}. On the other hand, since the projection operator $\mathcal{P}_{\mathbb{R}^n_+}:\mathbb{R}^n\rightarrow{\mathbb{R}^n_+}$ is continuous, letting $m\rightarrow+\infty$ in \eqref{26} and applying Theorem \ref{29} deduce that
   \begin{equation*}
   \begin{cases}
\widehat{u}(x)=\mathcal{P}_{\mathbb{R}^n_+}([(\mathbb{I}_{H^n}-\epsilon_0P){\widehat{u}}](x)-\epsilon_0Q(x)-\lim_{m\rightarrow+\infty}\lambda^{(m)}),\quad \mbox{a.e.} \;x\in[a,b],\\
\mbox{s.t.}\quad \langle \widehat{u},e_i\rangle_{H^n}=1,\quad\forall i\in N,
   \end{cases}
 \end{equation*}
where $\widehat{u}$ uniquely solves Problem \ref{12}. Therefore, $\lambda=\lim_{m\rightarrow+\infty}\lambda^{(m)}$. According to  Theorem \ref{59}, $\widehat{u}$ is   the unique NNE of Problem \ref{2}.
 \end{proof}
\begin{remark}\label{80}
It is worth mentioning that one may consider  the time-dependent density $u_i(t,x)$ $(\forall i\in N)$ if the firms are allowed to choose different production $x$ as  time varies, where $\int_a^bu_i(t,x)dx=1$ for any $t\in[s,T]$, and for each $i\in N$, $u_i(t,x)\geq0$ a.e. $x\in[a,b]$ and any $t\in[s,T]$. However, the techniques of this paper may not be available to such a problem because that the finite-dimensional variational inequality is expected to hold  in the sense of ``a.e. $t\in[s,T]$".

\begin{example}\label{58}
For given $g\in H$, let $(W_2g)(x)=2\int_0^1\cos(\pi(x-y))g(y)dy$ ($x\in[a,b]$),  then one can regard  $W_2$ as an  operator in $\mathcal{L}(H)$ ($W_2$ is also well-known as the graphon operator (\citep{carmona2022stochastic})). Consider the controlled state equation
\begin{equation*}
\begin{cases}
dX(t)=\left[f-u\right]dt,\quad t\in[0,1],\\
X(0)=\xi_0,
\end{cases}
\end{equation*}
where $u(x)$ stands for the density of choosing strategy $x$ in the single market, $\xi_0\in H$ is given, $f(x)=\frac{1}{2}\sin(\pi x)$ ($x\in[a,b]=[0,1]$), and the  payoffs of the single market
$$
V=J(u)=\beta-\frac{5\pi}{4}\cos(\pi x)+2\int_0^1W_2X(t)dt+\int_0^1u_1dt-2X(T),
$$
where $\beta$ is a given constant function in $H$. In the dynamic Cournot competition, the firms need to find a strategy $\widehat{u}\in U$ such that for a.e. $x\in[0,1]$,
\begin{equation}\label{43}
\widehat{u}(x)>0\Rightarrow \widehat{V}(x)\geq \left\langle \widehat{V},\widehat{u}\right\rangle_H.
\end{equation}
One can check that $Y(t)=-2tW_2$, $P=\mathbb{I}_{H}+W_2$ and $Q(x)=\beta-\frac{5\pi}{4}\cos(\pi x)-\frac{1}{2}\sin (\pi x)$.  According to Remark \ref{55}, it is easy to show that $P=\mathbb{I}_{H}+W_2$ is self-adjoint and thus $V=Pu+Q$ is a full potential game.
Noticing that $W_2^2=W_2$ and $\|Z\|^2_{H}\leq\langle PZ,Z\rangle_H\leq2\|Z\|^2_{H}$ $(\forall Z\in H)$, and applying Theorem \ref{29}, we can choose $\epsilon_0=\frac{1}{3}$. In this case, \eqref{30} becomes
  \begin{equation}\label{31}
   \begin{cases}
\widehat{u}(x)=\left(\left[\frac{2}{3}\widehat{u}(x)-\frac{1}{3}W{\widehat{u}}(x)\right]-\frac{1}{3}Q(x)-\lambda\right)\vee 0,\quad x\in[0,1],\\
\mbox{s.t.}\quad \int_0^1 \widehat{u}(x)dx=1.
   \end{cases}
 \end{equation}
With $\lambda=-\frac{1}{3}\beta$, by simple calculations one can show that $\widehat{u}:\widehat{u}(x)=\pi\cos(\pi x)\vee0$ ($x\in[0,1]$) is the unique solution to \eqref{31}. Therefore, according to Theorem \ref{29}, $\widehat{u}:\widehat{u}(x)=\pi\cos(\pi x)\vee0$ ($x\in[0,1]$)  is the unique NNE  of such an ID-NDG.
Indeed, we have
 $$
 \widehat{V}(x)=\beta+(-\pi\cos(\pi x))\vee0,\quad\left\langle \widehat{V},\widehat{u}\right\rangle_H=\beta\quad\mbox{and}\quad\widehat{V}(x)\geq \left\langle \widehat{V},\widehat{u}\right\rangle_H,\quad \forall x\in[0,1].
 $$
Therefore, inequality \eqref{43} can be satisfied.
\end{example}
\end{remark}

\subsection{Stability of   NNEs}
In this subssection, we discuss the stability of the NNEs to Problem \ref{12}. To this end, we assume that the sequences $\{\varepsilon_{1k},\varepsilon^{(i)}_{2k}\}_{k=1}^{+\infty}\subseteq \mathbb{R}$, $\{A_{ik},B_{ijk},E_{ik},F_{ik}\}_{k=1}^{+\infty}\subseteq C([s,T],\mathcal{L}(H))$, $\{\xi_{ik},f_{ik},\alpha_{ik}\}_{k=1}^{+\infty}\subseteq H$, and $\{G_{ik}\}_{k=1}^{+\infty}\subseteq \mathcal{L}(H)$  ($i\in N$) are all given. Consider the perturbed controlled state equations ($i\in N$),
\begin{equation}\label{46}
dX_{ik}(t)=[A_{ik}(t)X_{ik}(t)+\sum_{j\in N}B_{ijk}(t)v_{jk}+f_{ik}(t)]dt,\quad X_{ik}(s)=\xi_{ik}\in H,
\end{equation}
and the set
\begin{align}\label{79}
U^n_{\varepsilon}=\{v\in H^n|\langle v,e_i\rangle_{H^n}=1+\varepsilon_{1k}
\;\mbox{and}\;\langle v(x),e_i\rangle_{\mathbb{R}^n}\geq\varepsilon^{(i)}_{2k}\mbox{ a.e. }x\in[a,b]\;,\;\forall i\in N\}.
\end{align}
where ${v}_{ik}(x)$ stands for the density of choosing strategy $x$ in the $i$-th market. Then, the approximate ID-NDG problem can be defined as follows.
\begin{prob}\label{39}
Find an NNE  $\widehat{v}_k=(\widehat{v}_{1k},\cdots,\widehat{v}_{nk})^{\top} \in U^n_{\varepsilon}$ such  that for any $i\in N$ and a.e. $x\in[a,b]$,
$
\widehat{v}_{ik}(x)>0\quad\Rightarrow\quad\widehat{V}_{ik}(x)\geq\langle\widehat{V}_{ik},\widehat{v}_i\rangle_{H},
$
where $
\widehat{V}_{ik}=\alpha_{ik}+\int_s^T(E_{ik}(t)\widehat{X}_{ik}(t)+F_{ik}(t)\widehat{v}_{ik})dt+G_{ik}\widehat{X}_{ik}(T).
$
represents the optimal payoffs in the $i$-th market. Here   $\alpha_{ik}\in H$, $G_{ik}\in\mathcal{L}(H)$, $E_{ik},F_{ik}\in C([s,T],\mathcal{L}(H))$ are all given, and $\widehat{X}_{ik}(t)$ is governed by:
\begin{equation}\label{71}
d\widehat{X}_{ik}(t)=[A_{ik}(t)\widehat{X}_{ik}(t)+\sum_{j\in N}B_{ijk}(t)\widehat{v}_{jk}+f_{ik}(t)]dt,\quad \widehat{X}_{ik}(s)=\xi_{ik}\in H,
\end{equation}
\end{prob}
By similar arguments as in subsection \ref{40}, one can show that Problem \ref{39} is equivalent to the following Problem.
\begin{prob}\label{32}
Find $\widehat{v}_k=(\widehat{v}_{1k},\cdots,\widehat{v}_{nk})^{\top} \in U^n_{\varepsilon}$ such  that
\begin{equation}\label{35}
\langle P_k\widehat{v}_k+Q_k,Z-\widehat{v}_k\rangle_{H^n}\geq 0,\quad \forall Z\in U^n_{\varepsilon}.
\end{equation}
Here  the coefficients  $P_k$ and $Q_k$ are given by
\begin{equation*}
\begin{cases}
P_k:H^n\rightarrow H^n,\;P_k\widehat{v}_k=\int_s^TP_k(t)\widehat{v}_kdt, \quad \mathcal{P}_{ik}(t)=Y_{ik}(t)B_{iik}(t)+F_{ik}(t),\\
\overline{Y}_{ik}=\alpha_{ik}+Y_{ik}(s)\xi_{ik}+\int_s^TY_{ik}(t)f_{ik}(t)dt, \quad Q_k=(\overline{Y}_{1k},\cdots,\overline{Y}_{nk})^{\top},\\
P_k(t)=\begin{bmatrix}
\mathcal{P}_{1k}(t) & Y_{1k}(t)B_{12k}(t)&\dots & Y_{1k}(t)B_{1nk}(t)\\
  Y_{2k}(t)B_{21k}(t)& \mathcal{P}_{2k}(t)&\dots & Y_{2k}(t)B_{2nk}(t)\\
\vdots &\vdots & \ddots&\vdots\\
Y_{nk}(t)B_{n1k}(t)&Y_{nk}(t)B_{n2k}(t)&\dots & \mathcal{P}_{nk}(t)
\end{bmatrix}_{n \times n},
\end{cases}
\end{equation*}
where  $Y_{ik}\in C([s,T],\mathcal{L}(H))$ ($i\in N$) is the strong solution to the equation
\begin{equation}\label{72}
dY_{ik}(t)=-(E_{ik}(t)+Y_{ik}(t)A_{ik}(t))dt,\quad Y_{ik}(T)=G_{ik}.
\end{equation}
\end{prob}
By Lemma  \ref{8} and Theorem \ref{23}, the existence and uniqueness of $\widehat{X}_{ik}\in C([s,T],H)$ to \eqref{71} and $Y_{ik}\in C([s,T],\mathcal{L}(H))$ to \eqref{72} can be guaranteed. The stability of the NNE and the OWAP  of  the markets can be defined as follows.
\begin{defn}
The NNE   (or the OWAP) is said to be stable (or inherit the property of stability), if the following holds under Assumption \ref{21}:
\begin{equation}\label{73}
\lim_{k\rightarrow+\infty}\|\widehat{u}-\widehat{v}_k\|_{H^n}=0 \quad(\mbox{or } \langle \widehat{u},\widehat{V}\rangle_H=\lim_{k\rightarrow+\infty}\langle \widehat{v}_k,\widehat{W}_k\rangle_H),
\end{equation}
 where $\widehat{v}_k$ solves Problem \ref{32}, and $\widehat{W}_k=(\widehat{V}_{1k},\cdots,\widehat{V}_{nk})^{\top}=P_k\widehat{v}_k+Q_k.$
\end{defn}
\begin{remark}
The stability of the NNE and the OWAP  provide that both of them can be approximated   by their counterparts in Problem \ref{32} when $k$ is large enough.
\end{remark}
 In this subsection we  focus on answering  two questions: (A) Does there exists unique solution $\widehat{v}_k\in U^n_{\varepsilon}$ to Problem \ref{32} for a fixed $k$? (B) Can the NNE    and OWAP be stable? To this end,  we need the following assumption.
\begin{assumption}\label{21}
Let $B=[B_{ij}]_{i,j\in N}$ and $B_k=[B_{ijk}]_{i,j\in N}$. Assume that, for each $i\in N$ we have $\min_{k}\varepsilon_{1k}>-1$ and
\begin{equation*}
  \begin{cases}
  \lim_{k\rightarrow+\infty}\varepsilon_{1k}=\lim_{k\rightarrow+\infty}\varepsilon_{2k}^{(i)}=0,\quad\lim_{k\rightarrow+\infty}\|G_i-G_{ik}\|_{\mathcal{L}(H)}=0,\\
    \lim_{k\rightarrow+\infty}\|A_i-A_{ik}\|_{ C([s,T],\mathcal{L}(H))}=\lim_{k\rightarrow+\infty}\|B-B_{k}\|_{ C([s,T],\mathcal{L}(H^n))}=0, \\
  \lim_{k\rightarrow+\infty}\|\xi_i-\xi_{ik}\|_H=\lim_{k\rightarrow+\infty}\|f_i-f_{ik}\|_{H}=\lim_{k\rightarrow+\infty}\|\alpha_i-\alpha_{ik}\|_{H}=0, \\
  \lim_{k\rightarrow+\infty}\|E_i-E_{ik}\|_{ C([s,T],\mathcal{L}(H))}=\lim_{k\rightarrow+\infty}\|F_i-F_{ik}\|_{ C([s,T],\mathcal{L}(H))}=0.
  \end{cases}
\end{equation*}
\end{assumption}
Under Assumption \ref{21}, it is easy to see that    sequences of the norms on the operators    are all bounded in $\mathbb{R}$.   We also need the following two lemmas.
\begin{lemma}\label{14}
Under Assumption \ref{21}, $\lim_{k\rightarrow+\infty}\|Y_i-Y_{ik}\|_{ C([s,T],\mathcal{L}(H))}=0$.
\end{lemma}
\begin{proof}
 By Definition \ref{22}, for any $\eta\in H$ and all $t\in[s,T]$, we have
\begin{align*}
\|(Y_{i}(t)-Y_{ik}(t))\eta\|_{H}^2 
 \leq &4 \|G_i-G_{ik}\|_{\mathcal{L}(H)}^2\|\eta\|_H^2+4(T-s)\|E_i-E_{ik}\|^2_{C([s,T],\mathcal{L}(H))}\|\eta\|_H^2\\
 &+4(T-s)\|A_i-A_{ik}\|^2_{ C([s,T],\mathcal{L}(H))}\|Y_{ik}\|^2_{ C([s,T],\mathcal{L}(H))}\|\eta\|_H^2\\
   &+4\|A_i\|^2_{ C([s,T],\mathcal{L}(H))}\int_t^T\|Y_{i}(\tau)-Y_{ik}(\tau)\|^2_{\mathcal{L}(H)}d\tau\|\eta\|_H^2,
\end{align*}
and so for any $ t\in[s,T]$,
\begin{equation}\label{20}
\|Y_{i}(t)-Y_{ik}(t)\|_{\mathcal{L}(H)}^2\leq C_k+4\|A_i\|^2_{ C([s,T],\mathcal{L}(H))}\int_t^T\|Y_{i}(\tau)-Y_{ik}(\tau)\|^2_{\mathcal{L}(H)}d\tau,
\end{equation}
where $\lim_{k\rightarrow+\infty}C_k=0$. Then by Gronwall's inquality, for any $ t\in[s,T]$, one has
$$
\int_t^T\|Y_{k}(\tau)-Y_{ik}(\tau)\|^2_{\mathcal{L}(H)}d\tau\leq C_ke^{-4 t\|A_i\|^2_{ C([s,T],\mathcal{L}(H))}}\int_t^Te^{-4\tau\|A_i\|^2_{ C([s,T],\mathcal{L}(H))}}d\tau.
$$
Taking  limit on both sides gives
$\lim_{k\rightarrow+\infty}\int_t^T\|Y_{i}(\tau)-Y_{ik}(\tau)\|^2_{\mathcal{L}(H)}d\tau\leq0 \; (\forall t\in[s,T])$
and taking  limit again on both sides in \eqref{20} gives
$\lim_{k\rightarrow+\infty}\|Y_{i}(t)-Y_{ik}(t)\|_{\mathcal{L}(H)}^2\leq0 \; (\forall t\in[s,T]),$
which implies the conclusion.
\end{proof}

\begin{lemma}\label{33}
Under Assumption \ref{21}, one has $$\lim_{k\rightarrow+\infty}\|P-P_k\|_{\mathcal{L}(H^n)}=\lim_{k\rightarrow+\infty}\|Q-Q_k\|_{H^n}=0.$$
\end{lemma}
\begin{proof}
Let $Y(t)=\mbox{diag}(Y_1(t),\cdots,Y_n(t))$ and 
$Y_k(t)=\mbox{diag}(Y_{1k}(t),\cdots,Y_{nk}(t))$ and  set $\triangle F_{ik}=F_i-F_{ik}$,  $\triangle B_{k}=B-B_{k}$ and $\triangle Y_{k}=Y-Y_{k}$. Then clearly, we have
$$
\begin{cases}
Pv=\int_s^T[ (F_1(t)v_1,\cdots,F_n(t)v_n)^{\top}+Y(t)B(t)v]dt, \\
P_kv=\int_s^T[ (F_{1k}(t)v_{1k},\cdots,F_{nk}(t)v_{nk})^{\top}+Y_{k}(t)B_{k}(t)v]dt.
\end{cases}
$$
From Lemma \ref{14}, we know that $\lim_{karrow+\infty}\|Y-Y_k\|_{C([s,T],\mathcal{L}(H^n))}=0$. Therefore,   for any given $v\in H^n$, applying Cauchy-Schwarz's inequality  derives
\begin{align*}
&\|(P-P_k)v\|^2_{H^n}\\
=& \|\int_s^T[ (\triangle F_{1k}v_1,\cdots,\triangle F_{nk}v_n)^{\top}+\triangle Y_{k}(t)B(t)v+Y_{k}(t)\triangle B_{k}(t)v]dt\|^2_{H^n}\\
\leq& 3(T-s)\int_s^T\|(\triangle F_{1k}v_1,\cdots,\triangle F_{nk}v_n)^{\top}\|^2_{H^n}dt\\
&+3(T-s)\int_s^T\|\triangle Y_{k}(t)B(t)v\|^2_{H^n}dt+3(T-s)\int_s^T\|Y_{k}(t)\triangle B_{k}(t)v\|^2_{H^n}dt\\
\leq& 3(T-s)^2[\|v\|^2_{H^n}\max_{i\in N}\|\triangle F_{ik}\|^2_{ C([s,T],\mathcal{L}(H))}+\|\triangle Y_{k}\|^2_{ C([s,T],\mathcal{L}(H^n))}\|B\|^2_{C([s,T],\mathcal{L}(H))}\\
&
+\|Y_k\|^2_{C([s,T],\mathcal{L}(H))}\|\triangle B_{k}\|^2_{C([s,T],\mathcal{L}(H))}]\\
= &C_k\|v\|_{H^n}^2,
\end{align*}
where $\lim_{k\rightarrow+\infty}C_k=0$. Hence we obtain $\|P-P_k\|_{\mathcal{L}(H^n)}\leq \sqrt{C_k}$, which leads to $\lim_{k\rightarrow+\infty}\|P-P_k\|_{\mathcal{L}(H^n)}=0$.
Similarly, setting $\triangle \xi_{ik}=\xi_i-\xi_{ik}$,  $\triangle f_{ik}=f_i-f_{ik}$ and $\triangle Y_{ik}=Y_i-Y_{ik}$, one has
\begin{align*}
&\lim_{k\rightarrow+\infty}\|Q-Q_k\|^2_{H^n}\\
=&\lim_{k\rightarrow+\infty}\|(\alpha-\overline{\alpha}_{k})+(Y_1(s)\triangle\xi_{1k}+\triangle Y_{1k}(s)\xi_{1k},\cdots,Y_n(s)\triangle\xi_{nk}+\triangle Y_{nk}(s)\xi_{nk})^{\top}\\
&+\int_s^T(Y_1(t)\triangle f_{1k}+\triangle Y_{1k}(t)f_{1k},\cdots,Y_n(t)\triangle f_{nk}+\triangle Y_{nk}(t)f_{nk})^{\top}dt\|^2_{H^n}\\
\leq& 5n\lim_{k\rightarrow+\infty}[\max_{i\in N}\|\alpha_i-\overline{\alpha}_{ik}\|^2_H+\max_{i\in N}\|Y_i(s)\|^2_{\mathcal{L}(H)}\|\triangle\xi_{ik}\|^2_H\\
&+\max_{i\in N}\|\triangle Y_{ik}(s)\|^2_{\mathcal{L}(H)}\|\xi_{ik}\|^2_H+(T-s)^2\max_{i\in N}\|Y_i\|^2_{C([s,T],\mathcal{L}(H))}\max_{i\in N}\|\triangle f_{ik}\|^2_H\\
&+(T-s)^2\max_{i\in N}\|\triangle Y_{ik}\|^2_{C([s,T],\mathcal{L}(H))}\max_{i\in N}\|f_{ik}\|^2_H]
=0,
\end{align*}
which ends the proof.
\end{proof}
The corollary below gives an answer to question (A).
\begin{cor}\label{37}
Under Assumptions \ref{16} and \ref{21},  there exists a positive integer  $\mathcal{N}(\epsilon_P)$ such that $P_k$ satisfies  Assumption \ref{16} for any $k>\mathcal{N}(\epsilon_P)$.
\end{cor}
\begin{proof}
According to Lemma \ref{33}, there exists a positive integer  $\mathcal{N}(\epsilon_P)$ such that
$$
\|(P-P_k)Z\|^2_{H^n}\leq \frac{\epsilon_P^2}{4}\|Z\|^2_{H^n},\quad\forall Z\in H^n,\;\forall k>N.
$$  Therefore for any $k>N$, we have
$$
\langle (P-P_k)Z,Z\rangle_{H^n} \leq \frac{1}{\epsilon_P}\|(P-P_k)Z\|^2_{H^n}+\frac{\epsilon_P}{2}\|Z\|^2_{H^n}\leq \frac{3\epsilon_P}{4}\|Z\|^2_{H^n},\quad\forall Z\in H^n.
$$
Combining Assumption \ref{16} gives
$\langle P_kZ,Z\rangle_{H^n}\geq\frac{\epsilon_P}{4}\|Z\|^2_{H^n},\;\forall Z\in H^n$ and completes the proof.
\end{proof}

Corollary \ref{37} ensures the existence and uniqueness of solutions to Problem \ref{32} when $k>N$. Thus without loss of generality, we assume  that for each $ k\geq1$,
$\langle P_kZ,Z\rangle_{H^n}\geq \frac{\epsilon_P}{4}\|Z\|^2_{H^n}$ holds for arbitrary $Z\in H^n$,
where $\epsilon_P>0$ is given in Assumption \ref{16}, and so Problem \ref{32} admits a unique solution for any given $k\geq1$. Otherwise we should consider the new sequences beginning at $k=N+1$ instead in the sequel.

The main stability results can be stated as follows, which give a positive answer to question (B).
\begin{thm}\label{47}
Under Assumptions \ref{16},  both the NNE  and the OWAP are stable.
\end{thm}
\begin{proof}
Let
\begin{equation}\label{50}
\varepsilon_{2k}=(\varepsilon^{(1)}_{2k},\cdots,\varepsilon^{(n)}_{2k})^{\top}, \quad Z_k=h(Z)=\frac{Z-\varepsilon_{2k}}{1+\varepsilon_{1k}},\quad \overline{v}_k=h(\widehat{v}_k)=\frac{\widehat{v}_k-\varepsilon_{2k}}{1+\varepsilon_{1k}},
\end{equation}
where the mapping $h:U^n_{\varepsilon}\rightarrow U^n$ is  bijective, then one can check that $Z_k,\overline{v}_k\in U^n$ for any $k$, and
$\lim_{k\rightarrow+\infty}\|\varepsilon_{2k}\|_{H^n}^2=(b-a)\lim_{k\rightarrow+\infty}\|\varepsilon_{2k}\|_{\mathbb{R}^n}^2=0.$

It is clear that \eqref{35} can be rewritten as
\begin{equation}\label{36}
\langle P_k\overline{v}_k+P_k\frac{\varepsilon_{2k}}{1+\varepsilon_{1k}}+\frac{Q_k}{1+\varepsilon_{1k}},Z_k-\overline{v}_k\rangle_{H^n}\geq 0,\quad \forall Z_k\in U^n.
\end{equation}
According to Theorem \ref{29} and Corollary \ref{37}, when $k>N$,  solving ID-VI \eqref{36} is equivalent to solving the following projection equation with the constant 
\begin{equation}\label{38}
\overline{v}_k=\mathcal{P}_{U^n}[(\mathbb{I}_{H^n}-\epsilon_0P_k)\overline{v}_k-\epsilon_0(P_k\frac{\varepsilon_{2k}}{1+\varepsilon_{1k}}+\frac{Q_k}{1+\varepsilon_{1k}})].
\end{equation}
Here  $\epsilon_0\in(0,[\epsilon_P\|P\|_{\mathcal{L}(H^n)}^{-2}]\wedge [\epsilon_P\max_{k}\|P_k\|_{\mathcal{L}(H^n)}^{-2}])\subseteq(0,2\epsilon_P\|P\|_{\mathcal{L}(H^n)}^{-2})$ can be chosen the same in \eqref{17} and \eqref{38}. We claim that the sequence $\{\|\overline{v}_k\|_{H^n}\}_{k=1}^{+\infty}$ is bounded. Indeed, we have from Lemma \ref{33} that
$\{\|P_k\|_{\mathcal{L}(H^n)}\}_{k=1}^{+\infty}$ and $\{\|Q_k\|_{H^n}\}_{k=1}^{+\infty}$ are bounded. By choosing $\delta'>0$ with $\delta_1=(1+\delta')(1-\epsilon_0\epsilon_P)<1$ and observing the inequality 
\begin{equation}\label{41}
\|c+d\|_{H^n}^2=\|c\|_{H^n}^2+\|d\|_{H^n}^2+2\langle c,d\rangle_{H^n}\leq(1+\rho)\|c\|_{H^n}^2+(1+\frac{1}{\rho})\|d\|_{H^n}^2
\end{equation}
for arbitrary $\rho>0$, we have
\begin{align*}
\|\overline{v}_k\|_{H^n}^2&=\|\mathcal{P}_{U^n}[(\mathbb{I}_{H^n}-\epsilon_0P_k)\overline{v}_k-\epsilon_0(P_k\frac{\varepsilon_{2k}}{1+\varepsilon_{1k}}+\frac{Q_k}{1+\varepsilon_{1k}})]\|_{H^n}^2\\
&\le \delta_1\|\overline{v}_k\|_{H^n}^2+\frac{2\epsilon_0^2(1+\delta')}{\delta'(1+\varepsilon_{1k})^2}(\|P_k\|_{\mathcal{L}(H^n)}^2
\|\varepsilon_{2k}\|_{H^n}^2+\|Q_k\|_{H^n}^2),
\end{align*}
where we have used the non-expansive property of metric projection operators and the fact that
$
\|(\mathbb{I}_{H^n}-\epsilon_0P_k)\overline{v}_k\|_{H^n}^2
\leq (1-\epsilon_0\epsilon_P))\|\overline{v}_k\|_{H^n}^2$ when $\epsilon_0\leq \epsilon_P\max_{k}\|P_k\|_{\mathcal{L}(H^n)}^{-2}$.
Hence
$$
\|\overline{v}_k\|_{H^n}^2\leq \max_{k}\frac{2\epsilon_0^2(1+\delta')}{\delta'(1-\delta_1)(1+\varepsilon_{1k})^2}[\max_{k}\|P_k\|_{\mathcal{L}(H^n)}^2\max_{k}\|\varepsilon_{2k}\|_{H^n}^2+\max_{k}\|Q_k\|_{H^n}^2],
$$
and the boundedness of $\{\|\overline{v}_k\|_{H^n}\}_{k=1}^{+\infty}$ follows.

Next, by choosing $\delta>0$ with $\delta_2=(1+\delta)(1-\epsilon_0\epsilon_P)<1$, from \eqref{41} one has
\begin{align*}
\|\widehat{u}-\overline{v}_k\|_{H^n}^2
\leq & \delta_2\|\widehat{u}-\overline{v}_k\|_{H^n}^2+4\epsilon^2_0(1+\frac{1}{\delta})\|P-P_k\|_{\mathcal{L}(H^n)}^2\|\overline{v}_k\|_{H^n}^2\\
&+\frac{4\epsilon^2_0(1+\frac{1}{\delta})}{(1+\varepsilon_{1k})^2}[\|P_k\|_{\mathcal{L}(H^n)}^2\|\varepsilon_{2k}\|_{H^n}^2+\|Q_k-Q\|_{H^n}^2+\varepsilon^2_{1k}\|Q\|_{H^n}^2],
\end{align*}
where we have also used the non-expansive property of metric projection operators and the fact that 
$\|(\mathbb{I}_{H^n}-\epsilon_0P_k)\overline{v}_k\|_{H^n}^2
=\|\overline{v}_k\|_{H^n}^2-2\epsilon_0\langle P_k\overline{v}_k,\overline{v}_k\rangle_{H^n}+
\epsilon_0^2\|P_k\overline{v}_k\|_{H^n}^2
\leq (1-\epsilon_0\epsilon_P)\|\overline{v}_k\|_{H^n}^2.
$
 Applying Lemma \ref{33} and the boundedness of $\{\|\overline{v}_k\|_{H^n}\}_{k=1}^{+\infty}$ gives
$\|\widehat{u}-\overline{v}_k\|_{H^n}^2\leq C_k$,
 where $\lim_{k\rightarrow+\infty}C_k=0$.

On the other hand, we have from \eqref{50} that
$
\|\widehat{u}-\overline{v}_k\|_{H^n}
\geq\|\widehat{u}-\widehat{v}_k\|_{H^n}-\frac{\varepsilon_{1k}}{1+\varepsilon_{1k}}\|\widehat{v}_k\|_{H^n}
-\|\frac{\varepsilon_{2k}}{1+\varepsilon_{1k}}\|_{H^n}.
$
Thus, $\|\widehat{u}-\widehat{v}_k\|_{H^n}\leq\frac{\varepsilon_{1k}}{1+\varepsilon_{1k}}\|\widehat{v}_k\|_{H^n}
+\frac{1}{1+\varepsilon_{1k}}\left\|\varepsilon_{2k}\right\|_{H^n}+\sqrt{C_k}.$
Letting $k\rightarrow+\infty$ on both sides gives $\lim_{k\rightarrow+\infty}\|\widehat{u}-\widehat{v}_k\|_{H^n}=0$.

For the OWAP, using Lemma \ref{33} and the boundedness of $\{\|\overline{v}_k\|_{H^n}\}_{k=1}^{+\infty}$, we have
\begin{align}\label{42}
&\lim_{k\rightarrow+\infty}\|\widehat{V}-\widehat{W}_k\|_{H^n}
= \lim_{k\rightarrow+\infty}\|P\widehat{u}+Q-P_k\widehat{v}_k-Q_k\|_{H^n}\\   
   \leq&\lim_{k\rightarrow+\infty}(\|P\|_{\mathcal{L}(H^n)}\|\widehat{u}-\widehat{v}_k\|_{H^n}+\|P-P_k\|_{\mathcal{L}(H^n)}\|\widehat{v}_k\|_{H^n}
   +\|Q-Q_k\|_{H^n})=0.\nonumber
\end{align}
It is clearly that
$$
|\langle \widehat{u},\widehat{V}\rangle_H-\langle \widehat{v}_k,\widehat{W}_k\rangle_H|
\leq\|\widehat{u}-\widehat{v}_k\|_{H^n}\|\widehat{V}\|_{H^n}+\|\widehat{v}_k\|_{H^n}\|\widehat{V}-\widehat{W}_k\|_{H^n}.
$$
Taking limit on both sides and by Theorem \ref{47}, $\lim_{k\rightarrow+\infty}|\langle \widehat{u},\widehat{V}\rangle_H-\langle \widehat{v}_k,\widehat{W}_k\rangle_H|\leq 0,$
where we have used the  boundedness of $\{\|\overline{v}_k\|_{H^n}\}_{k=1}^{+\infty}$. Thus the result follows.
\end{proof}
\begin{remark}\label{90}
Similarly, we can define the stability of the value function, i.e., $\lim_{k\rightarrow+\infty}\|\widehat{V}-\widehat{W}_k\|_H=0$ holds under Assumption \ref{16}. From \eqref{42}, the stability of the value function is provided.
\end{remark}
To illustrate  the results obtained in this section, we provide the following example.
\begin{example}\label{67}
Recalling the operators $F$ and $W_1$ defined in Example \ref{75}, one can check that the operators $W_1\in\mathcal{L}(H)$, $W_1F\in \mathcal{L}(H)$ are self-adjoint. We in addition suppose that $s=0$, $a>0$, $\varsigma_5>0$, $\varsigma_0,\varsigma_6\in H$, $\varsigma_7\in C([0,T],H)$ and that
$\varsigma_0(x)>\varsigma_2b$, $\varsigma_6(x)>-\varsigma_4b$  and $\xi(x)>0$ for a.e. $x\in[a,b]$.
From  the facts that $Fu=\int_a^bxu(x)dx\in[a,b]$ and that $u(x)\geq0$ a.e. $x\in[a,b]$, we can deduce that the demand and the price are positive, i.e.,
$X(t,x)>0$ and $p(t,x)>0$ for any  $t\in[0,T]$ and a.e.  $x\in[a,b]$. Then the payoffs can be rewritten as
$V=\int_0^T[\varsigma_3W_1X(t)+\varsigma_4W_1Fu+\varsigma_5W_1u+W_1\varsigma_6(t)-W_1\varsigma_7(t)]dt.$
The firms aim to find $\widehat{u} \in U$  such  that for a.e. $x\in[a,b]$, $\widehat{u}(x)>0\Rightarrow\widehat{V}(x)\geq\langle\widehat{V},\widehat{u}\rangle_{H}$.

We would like to discuss this problem in  two cases: (I) $\varsigma_1\neq0$ and $\frac{b^3-a^3}{3}\theta_1+\varsigma_5aT>0$ with $\theta_1=(\frac{\varsigma_2\varsigma_3}{\varsigma_1}+\varsigma_4)T-\frac{\varsigma_2\varsigma_3}{\varsigma^2_1}e^{\varsigma_1T}$;
(II) $\varsigma_1=0$ and $\frac{b^3-a^3}{3}\theta_2+\varsigma_5aT>0$ with where $\theta_2=\varsigma_4T-\frac{\varsigma_2\varsigma_3}{2}T^2$.

In  case (I), one can check that
$Y(t)=\frac{\varsigma_3}{\varsigma_1}(e^{\varsigma_1(T-t)}-1)W_1$, $P=\theta_1 W_1F+\varsigma_5TW_1$, and
$Q=\int_0^TW_1(\varsigma_6(t)-\varsigma_7(t))dt+\frac{\varsigma_3}{\varsigma_1}(e^{\varsigma_1T}-1)W_1\xi+(\frac{\varsigma_3}{\varsigma^2_1}e^{\varsigma_1T}-\frac{\varsigma_3T}{\varsigma_1})W_1\varsigma_0.$
If $\theta_1\geq0$, we have $\langle Pg,g\rangle_H=\theta_1\|Fg\|_H^2+\varsigma_5T\int_a^bxg^2(x)dx\geq \varsigma_5aT\|g\|_H^2,\; \forall g\in H.$
If $\theta_1<0$, we have from Cauchy-Schwarz's inequality that
$
\langle Pg,g\rangle_H\geq \theta_1\|g\|_H^2\int_a^bx^2dx+\varsigma_5aT\|g\|_H^2=(\frac{b^3-a^3}{3}\theta_1+\varsigma_5aT)\|g\|_H^2,\; \forall g\in H.
$
Thus, Assumption \ref{16} is satisfied. According to Remark \ref{55}, it is easy to show that the operator $P=\theta_1 W_1F+\varsigma_5TW_1$ is self-adjoint and thus $V=Pu+Q$ is a full potential game. By Theorems \ref{29}, \ref{18} and \ref{49}, the NNE $\widehat{u}$ can be obtained by the iteration \eqref{26}. And by Theorem \ref{47}, both the NNE  and the OWAP  are stable.

In case (II), one can also check that $Y(t)=\varsigma_3(T-t)W_1$, $ P=\theta_2 W_1F+\varsigma_5TW_1$, and
$Q=\int_0^TW_1(\varsigma_6(t)-\varsigma_7(t))dt+\varsigma_3TW_1\xi+\frac{\varsigma_3}{2}T^2W_1\varsigma_0$.
Similarly,  it can be shown that Assumption \ref{16} is satisfied and that  $V=Pu+Q$ is a full potential game. By Theorems \ref{29}, \ref{18} and \ref{49}, the NNE $\widehat{u}$ can be obtained by the iteration \eqref{26}. And by Theorem \ref{47},  all of the NNE, the value function,  and the OWAP inherit the property of stability.

Set the parameters as $[a,b]=[1,2]$, $s=\varsigma_1=\varsigma_4=0$, $\varsigma_2=0.5$, $\varsigma_3=\varsigma_5=2$,  and $\varsigma_6=\varsigma_7=T=1$. Then by the iteration \eqref{26}, we can obtain Figures \ref{Fig.1}-\ref{Fig.4},  which demonstrate the stability of the  NNE,  the value function,    and the OWAP   with respect to variations in $(\varsigma_0,\xi)$  and the convergence of $\lambda$. These numerical results align with the theoretical predictions provided by Theorem \ref{47}\, further validating the robustness of the NNE, the value function, and OWAP under the specified conditions in the framework of ID-NDG, and indicating that an increase in demand may lead to an increase in payoffs.
\begin{figure}
  \centering
  \includegraphics[height=4.5cm, width=12cm]{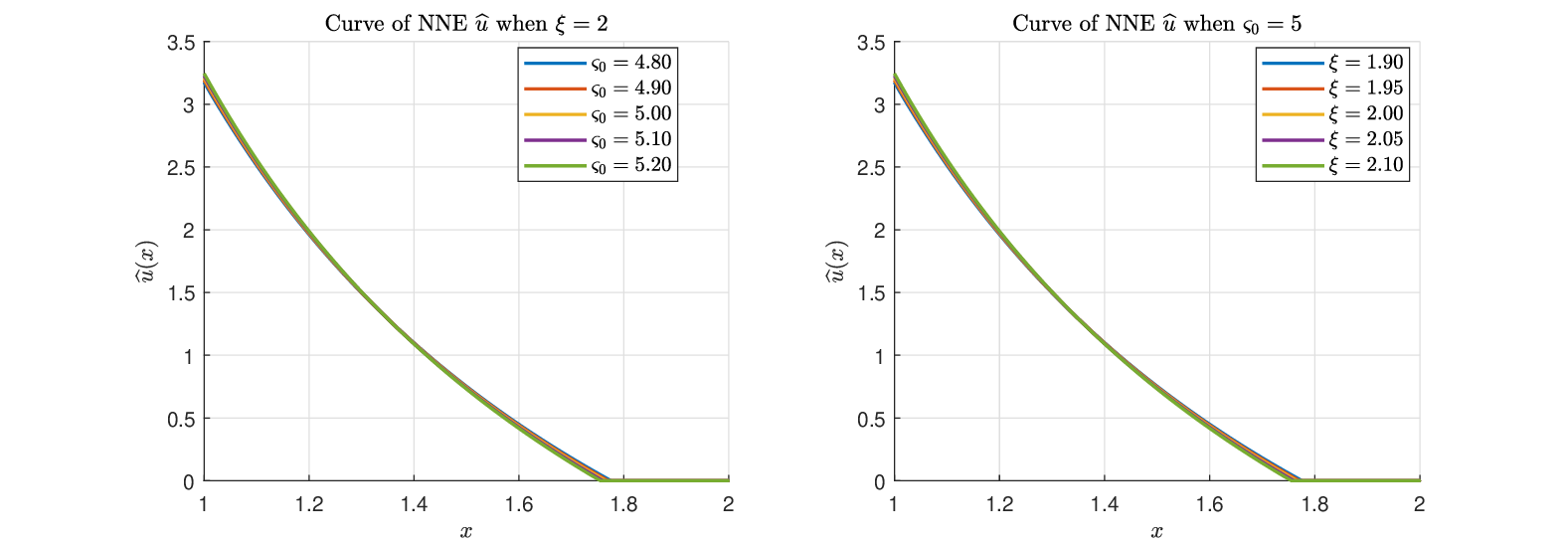}
  \caption{Stability of the NNE}\label{Fig.1}
\end{figure}

\begin{figure}
  \centering
  \includegraphics[height=4.5cm, width=12cm]{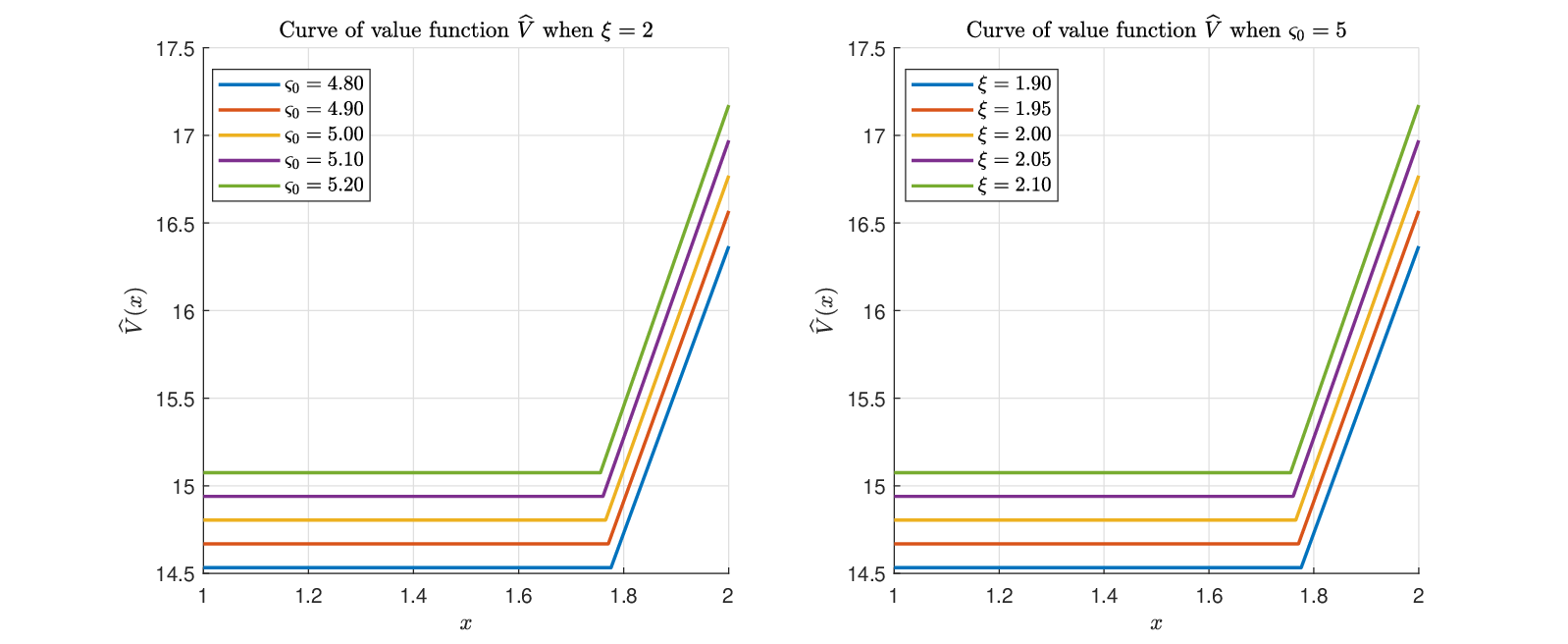}
  \caption{Stability of  the value function }\label{Fig.2}
\end{figure}

\begin{figure}
  \centering
  \includegraphics[height=4.5cm, width=12cm]{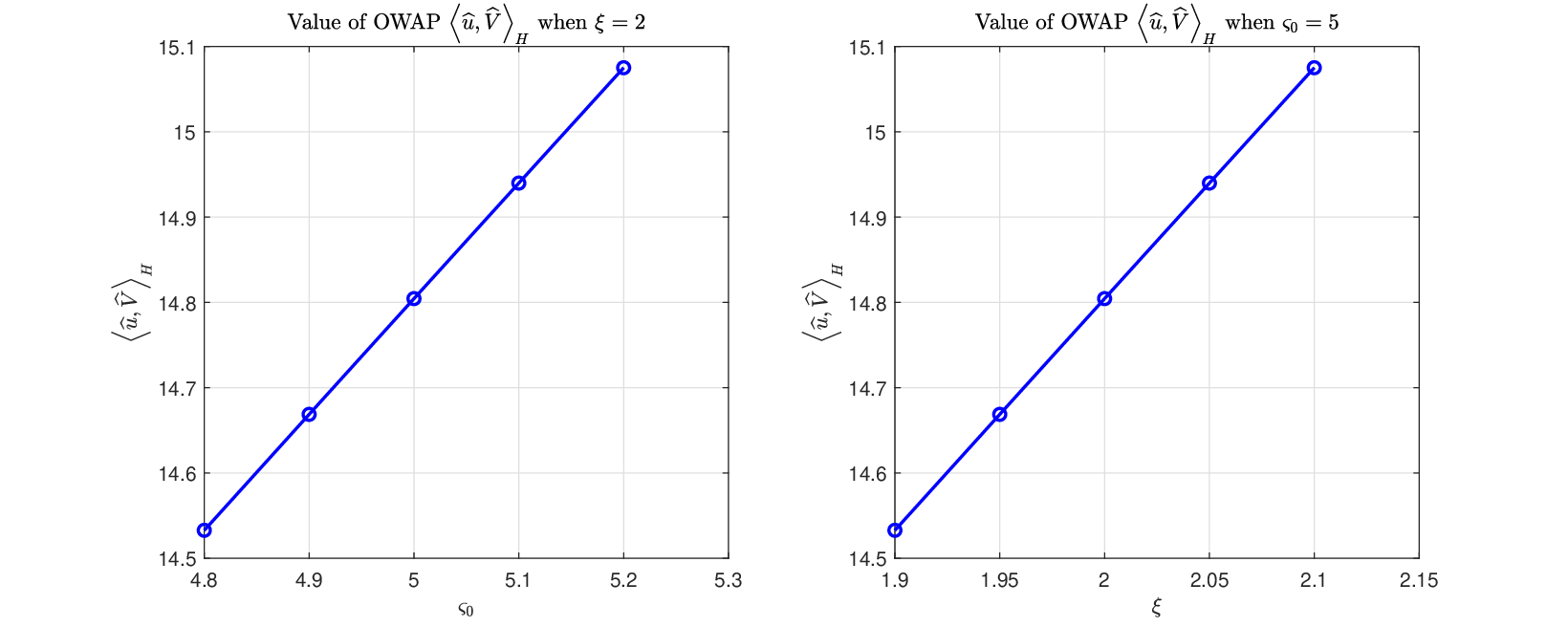}
  \caption{Stability of  the OWAP}\label{Fig.3}
\end{figure}

\begin{figure}
  \centering
  \includegraphics[height=5cm, width=14cm]{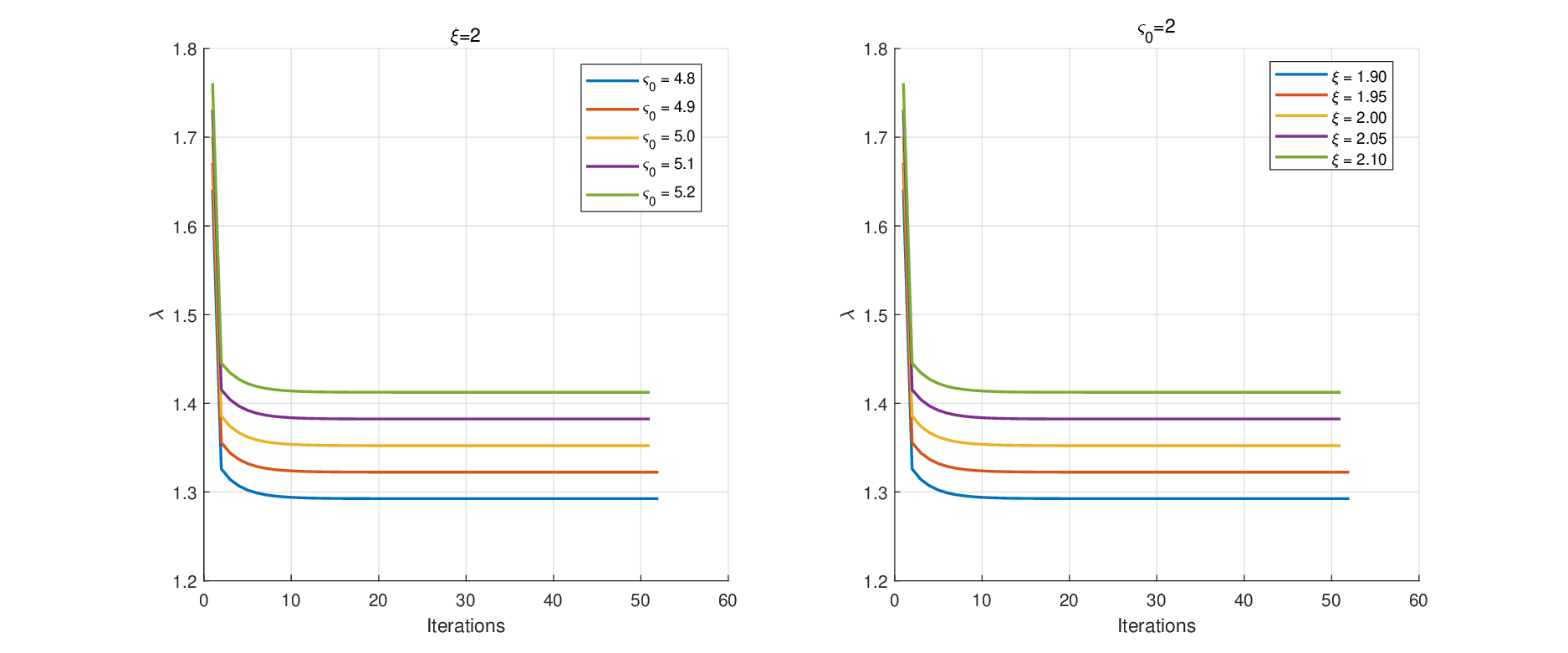}
  \caption{Convergence of $\lambda$}\label{Fig.4}
\end{figure}
\end{example}

\section{Existence, Uniqueness and Stability of  MNEs}
In this section, we discuss  existence, uniqueness and stability of MNEs to Problem \ref{62}. According to \eqref{60} and $V_i=J_i(u_i)$, we have
\begin{align*}
\langle \widehat{u}_i,J_i(\widehat{u}_i)\rangle_H=& \max_{u_i\in U}[\langle u_i,\alpha_i+Y_i(s)\xi_i+\int_s^TY_i(t)f_i(t)dt\rangle_H+ \langle u_i,\int_s^TF_i(t)u_i dt\rangle_H\\
  & +\langle u_i,\int_s^TY_i(t)(\sum_{j\neq i,j\in N}B_{ij}(t)\widehat{u}_j+B_{ii}(t)u_i)dt\rangle_H].
\end{align*}
It is well known that $\widehat{u}_i$ ($\forall i\in N$) is the solution to the following ID-VI:
\begin{equation}\label{68}
\langle \overline{Y}_{i}+\int_s^TY_i(t)\sum_{j\in N}B_{ij}(t)\widehat{u}_jdt
+\int_s^T\mathfrak{P}_{i}(t)\widehat{u}_i dt,Z-\widehat{u}_i\rangle_H\leq0,\quad \forall Z\in U,
\end{equation}
where $\overline{Y}_{i}=\alpha_i+Y_{i}(s)\xi_{i}+\int_s^TY_{i}(t)f_{i}(t)dt$ and $\mathfrak{P}_{i}(t)=Y_i(t)B_{ii}(t)+(B_{ii}(t)Y_i(t))^*+F_i(t)+F^*_i(t)$. Here $Y_i$ is the solution to \eqref{7}. Replace the expression of $P(t)$ in \eqref{15} by
$$
P(t)=\begin{bmatrix}
\mathfrak{P}_1(t) & Y_1(t)B_{12}(t)&\dots & Y_1(t)B_{1n}(t)\\
  Y_2(t)B_{21}(t)& \mathfrak{P}_2(t)&\dots & Y_2(t)B_{2n}(t)\\
\vdots &\vdots & \ddots&\vdots\\
Y_n(t)B_{n1}(t)&Y_n(t)B_{n2}(t)&\dots & \mathfrak{P}_n(t)
\end{bmatrix}_{n \times n}
$$
instead. Then by similar arguments as  in Theorem \ref{59}, one can show that solving ID-VI \eqref{68}   is equivalent to solving the following problem.
\begin{prob}\label{69}
Find the MNE $\widehat{u}=(\widehat{u}_1,\cdots,\widehat{u}_n)\in U^n$ such that
$\langle P\widehat{u}+Q,Z-\widehat{u}\rangle_{H^n}\leq 0$ for all $Z\in U^{n}$.
\end{prob}

On the other hand, consider the following problem.
\begin{prob}\label{70}
Find $\widehat{v}_k\in U^n_{\varepsilon}$ such  that
$
\langle P_k\widehat{v}_k+Q_k,Z-\widehat{v}_k\rangle_{H^n}\leq 0$  for all $Z\in U^n_{\varepsilon}$, where the set $U^n_{\varepsilon}$ is given by \eqref{79}, the coefficients  $P_k$ and $Q_k$ are given by
\begin{equation*}
\begin{cases}
P_k\widehat{u}=\int_s^TP_k(t)\widehat{u}dt,\quad Q_k=(\overline{Y}_{1k},\cdots,\overline{Y}_{nk})^{\top},\\
\mathcal{P}_{ik}(t)=Y_{ik}(t)B_{iik}(t)+B^*_{iik}(t)Y^*_{ik}(t)+F_{ik}(t)+F^*_{ik}(t),\\
P_k(t)=\begin{bmatrix}
\mathcal{P}_{1k}(t) & Y_{1k}(t)B_{12k}(t)&\dots & Y_{1k}(t)B_{1nk}(t)\\
  Y_{2k}(t)B_{21k}(t)& \mathcal{P}_{2k}(t)&\dots & Y_2(t)B_{2nk}(t)\\
\vdots &\vdots & \ddots&\vdots\\
Y_{nk}(t)B_{n1k}(t)&Y_{nk}(t)B_{n2k}(t)&\dots & \mathcal{P}_{nk}(t)
\end{bmatrix}_{n \times n},
\end{cases}
\end{equation*}
and the coefficients $Y_{ik}(t)$, $\overline{Y}_{ik}$, $B_{ik}(t)$ and $F_{ik}(t)$ are introduced in Section 4.
\end{prob}
We call the  MNE (or the OWAP)  stable if \eqref{73} holds true under Assumption \ref{21} (see Remark \ref{90} for the stability of the value function). By observing the inherent symmetry between Problems \ref{12} and \ref{69} (as well as Problems \ref{32} and \ref{70}), and building upon the analytical framework developed in Sections 3 and 4, we rigorously formalize the following result through analogous proof techniques.
\begin{thm}\label{65}
Suppose that $P$ is negative definite, i.e., $\langle PZ,Z\rangle_{H^n}\leq -\epsilon'_P\|Z\|^2_{H^n}$  ($\forall Z\in H^{n}$)
for some constant $\epsilon'_P>0$. Then there exists a unique solution $\widehat{u}=(\widehat{u}_1,\cdots,\widehat{u}_n)\in U^n$ to Problem \ref{62}, which can be solved by the iteration below:
  \begin{equation}\label{53}
   \begin{cases}
u^{(m+1)}(x)=\mathcal{P}_{\mathbb{R}^n_+}([(\mathbb{I}_{H^n}+\epsilon_0P){u^{(m)}}](x)+\epsilon_0Q(x)-\lambda^{(m)})\; \mbox{a.e.} \;x\in[a,b],\\
\mbox{s.t.}\quad \langle u^{(m+1)},e_i\rangle_{H^n}=1,\quad\forall i\in N,
   \end{cases}
 \end{equation}
 where  $\epsilon_0\in(0,2\epsilon'_P\|P\|_{\mathcal{L}(H^n)}^{-2})$, and the limit $\lambda=\lim_{m\rightarrow+\infty}\lambda^{(m)}$  exists.  Moreover, both   MNE  and   OWAP   are stable.
\end{thm}

To illustrate the results obtained in this section, we provide the following example.
\begin{example}\label{66}
Recalling the operators $F$ and $W_1$ defined in Example \ref{75}, one can check that the operators can check that the operators $W_1\in\mathcal{L}(H)$, $W_1F\in \mathcal{L}(H)$ are self-adjoint.  We in addition suppose that $s=0$,  $a>0$, $\varsigma_5<0$, $\varsigma_0,\varsigma_6\in H$, $\varsigma_7\in C([0,T],H)$ and that
$\varsigma_0(x)>\varsigma_2b$ and $\xi(x)>0$ for a.e.$x\in[a,b].$
From  the fact that $Fu=\int_a^bxu(x)dx\in[a,b]$ and $u(x)\geq0$ a.e. $x\in[a,b]$, we can deduce that the demand is positive, i.e.,
$X(t,x)>0$ for any $t\in[0,T]$ and a.e. $x\in[a,b]$. Then the payoffs can be rewritten as
$V=\int_0^T[\varsigma_3W_1X(t)+\varsigma_4W_1Fu+\varsigma_5W_1u+W_1\varsigma_6(t)-W_1\varsigma_7(t)]dt.$
The firms aim to find $\widehat{u} \in U$  such  that
$\langle \widehat{u},J(\widehat{u})\rangle_H=\max_{u\in U}\langle u,J(u)\rangle_H.$

We would like to discuss this problem in  two cases:
(I)  $\varsigma_1\neq0$ and $\frac{b^3-a^3}{3}\theta_1+\varsigma_5aT<0$, where $\theta_1=(\frac{\varsigma_2\varsigma_3}{\varsigma_1}+\varsigma_4)T-\frac{\varsigma_2\varsigma_3}{\varsigma^2_1}e^{\varsigma_1T}$; (II)  $\varsigma_1=0$ and $\frac{b^3-a^3}{3}\theta_2+\varsigma_5aT<0$, where $\theta_2=\varsigma_4T-\frac{\varsigma_2\varsigma_3}{2}T^2$.

In  case (I), one can check that
$Y(t)=\frac{\varsigma_3}{\varsigma_1}(e^{\varsigma_1(T-t)}-1)W_1$, $P=2\theta_1 W_1F+2\varsigma_5TW_1$, and  
$Q=\int_0^TW_1(\varsigma_6(t)-\varsigma_7(t))dt+\frac{\varsigma_3}{\varsigma_1}(e^{\varsigma_1T}-1)W_1\xi+(\frac{\varsigma_3}{\varsigma^2_1}e^{\varsigma_1T}-\frac{\varsigma_3T}{\varsigma_1})W_1\varsigma_0$.
By similar arguments as in Example \ref{66}, one can show that the negative definite condition on $P$ is satisfied.  By Theorem \ref{65}, the MNE $\widehat{u}$ can be obtained by the iteration \eqref{53}, and both the MNE and the OWAP are stable.

In case (II), one can also check that
$Y(t)=\varsigma_3(T-t)W_1$, $ P=2\theta_2 W_1F+2\varsigma_5TW_1$, and 
$Q=\int_0^TW_1(\varsigma_6(t)-\varsigma_7(t))dt+\varsigma_3TW_1\xi+\frac{\varsigma_3}{2}T^2W_1\varsigma_0$,
and that the negative definite condition on $P$ is satisfied. By Theorem \ref{65}, the MNE $\widehat{u}$ can be obtained by the iteration formula \eqref{53}, and all of the NNE, the value function and the OWAP   inherit the property of stability.

Set the parameters as $[a,b]=[1,2]$, $s=\varsigma_1=0$, $\varsigma_2=0.5$, $\varsigma_3=2$, $\varsigma_4=\varsigma_5=-\frac{1}{2}$,  and $\varsigma_6=\varsigma_7=T=1$. By the iteration \eqref{53}, we can obtain   Figures \ref{Fig.5}-\ref{Fig.8}, which demonstrate the stability of the  NNE, the value function,    and the OWAP  with respect to variations in $(\varsigma_0,\xi)$ and the convergence of $\lambda$. These numerical results align  with the theoretical predictions provided by Theorem \ref{65}, further validating the robustness of the NNE,  the value function,   and the  OWAP under the specified conditions in the framework of ID-MDG, and indicating that an increase in demand may lead to an increase in payoffs.
\begin{figure}
  \centering
  \includegraphics[height=4.5cm, width=12cm]{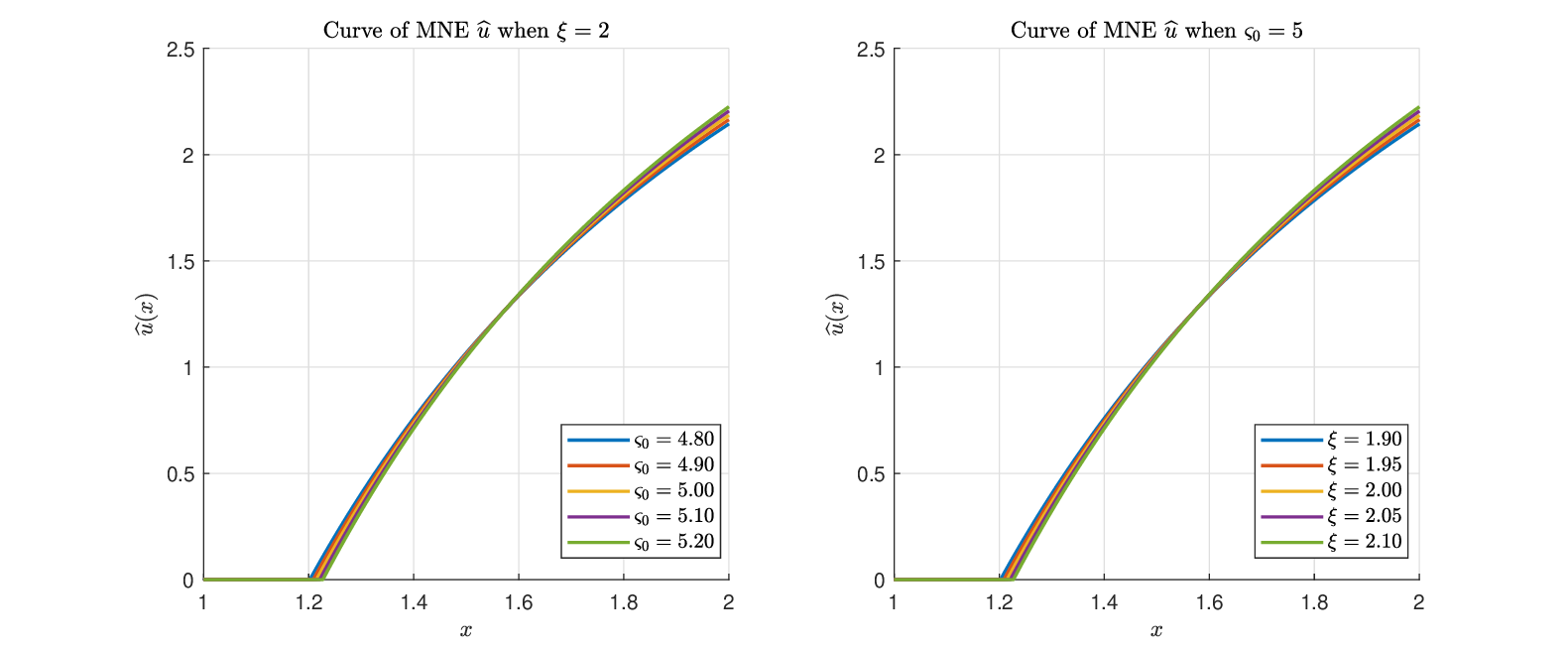}
  \caption{Stability of the NNE}\label{Fig.5}
\end{figure}

\begin{figure}
  \centering
  \includegraphics[height=4.5cm, width=12cm]{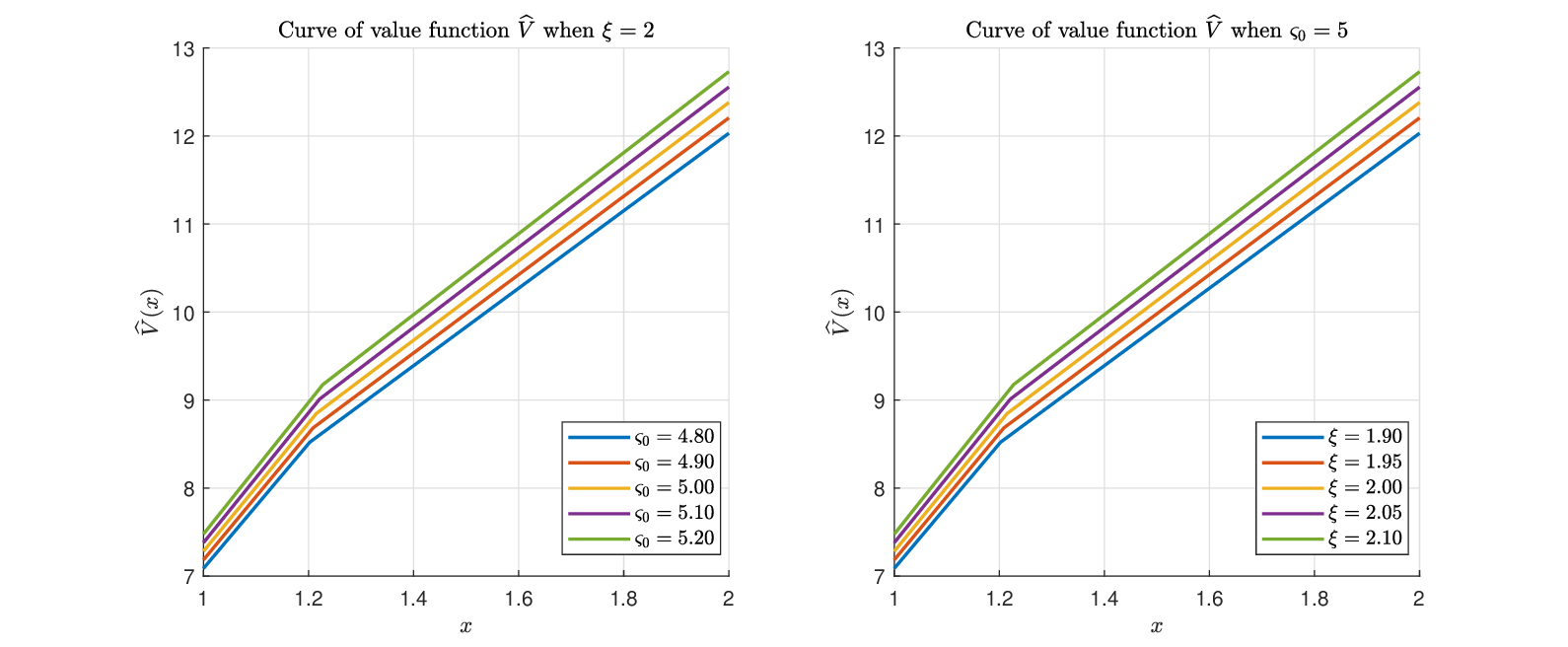}
  \caption{Stability of  the value function }\label{Fig.6}
\end{figure}

\begin{figure}
  \centering
  \includegraphics[height=4.5cm, width=12cm]{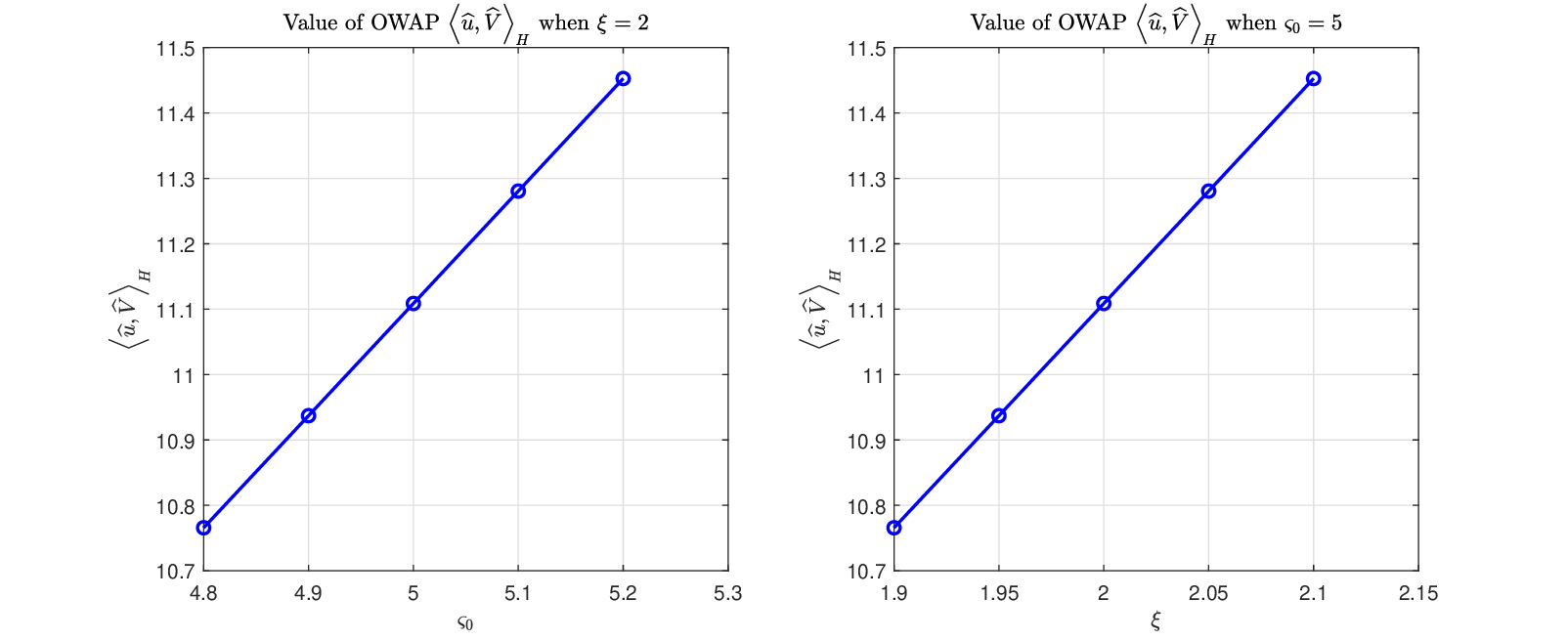}
  \caption{Stability of  the OWAP}\label{Fig.7}
\end{figure}
  
  \begin{figure}
  \centering
  \includegraphics[height=5cm, width=14cm]{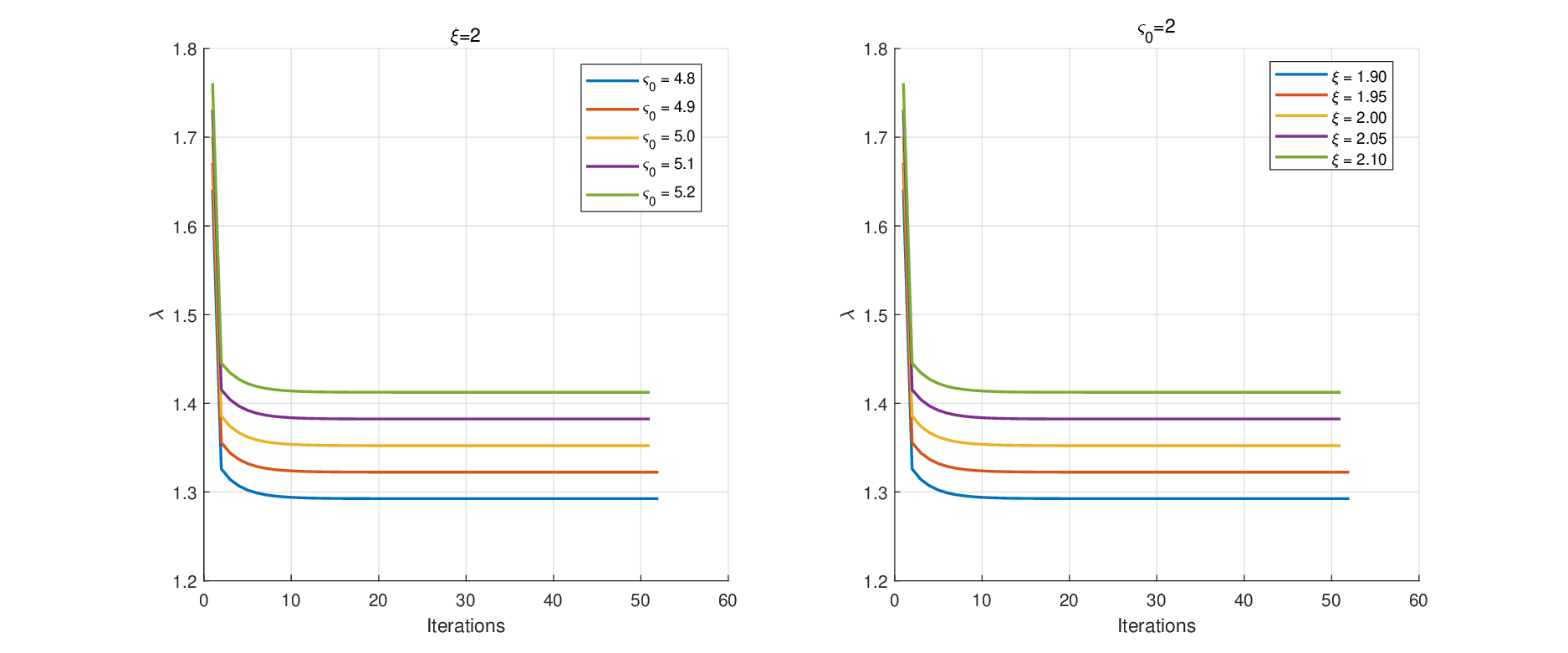}
  \caption{Convergence of $\lambda$}\label{Fig.8}
\end{figure}

\end{example}

\begin{remark}\label{77}
Examples \ref{67} and \ref{66} provide symmetry-dependent conditions involving the parameter $\varsigma_5$ and the term $\frac{b^3-a^3}{3}\theta_i+\varsigma_5aT$ $(i=1,2)$  that govern firms' strategies between competition and cooperation. Figure \ref{Fig.9} illustrates the following three distinct scenarios: (i) NNE: For $\varsigma_5=4$ and $(\varsigma_2,\varsigma_3,\varsigma_4)\in D_1\cup D_3$, firms adopt competitive strategies, resulting in a unique NNE within the ID-NDG framework; (ii) MNE: For $\varsigma_5=-4$ and $(\varsigma_2,\varsigma_3,\varsigma_4)\in D_2\cup D_3$,  cooperation emerges as the optimal strategy, with a unique MNE guaranteed under the ID-MDG framework; (iii) When  $\varsigma_5=-4$ and $(\varsigma_2,\varsigma_3,\varsigma_4)\in D_1$ (respectively,  $\varsigma_5=4$ and $(\varsigma_2,\varsigma_3,\varsigma_4)\in D_2$),   the NNE of ID-NDG (respectively, the MNE of ID-MDG) cannot be  certain to exist, as equilibrium existence fails in these parameter configurations.

Examples \ref{67} and \ref{66} also reveal critical limitations in the  specific parameter configurations of the price function: (i) When $\varsigma_5<0$ the price function may lose positivity due to the unbounded nature of  $u$; (ii) When $\varsigma_5=0$,  the operator $P=\theta_1 W_1F$ reduces to a non-negative definite form  in Example \ref{67} (respectively, non-positive definite form in Example \ref{66}), satisfying $\langle PZ,Z\rangle_{H^n}\geq0\mbox{ (respectively, $\leq0$)}$ for all $Z\in H^{n}$. This weaker definiteness invalidates the strict positive/negative definiteness assumptions required for equilibrium analysis. Thus the case $\varsigma_5=0$ raises two unresolved questions:  Do there exist  NNEs and MNEs to Examples \ref{67} and \ref{66}?  And what methodologies can reliably compute such equilibria when standard definiteness conditions fail? These two questions remain open and warrant dedicated theoretical investigation.

\begin{figure}
  \centering
  \includegraphics[height=5cm, width=12cm]{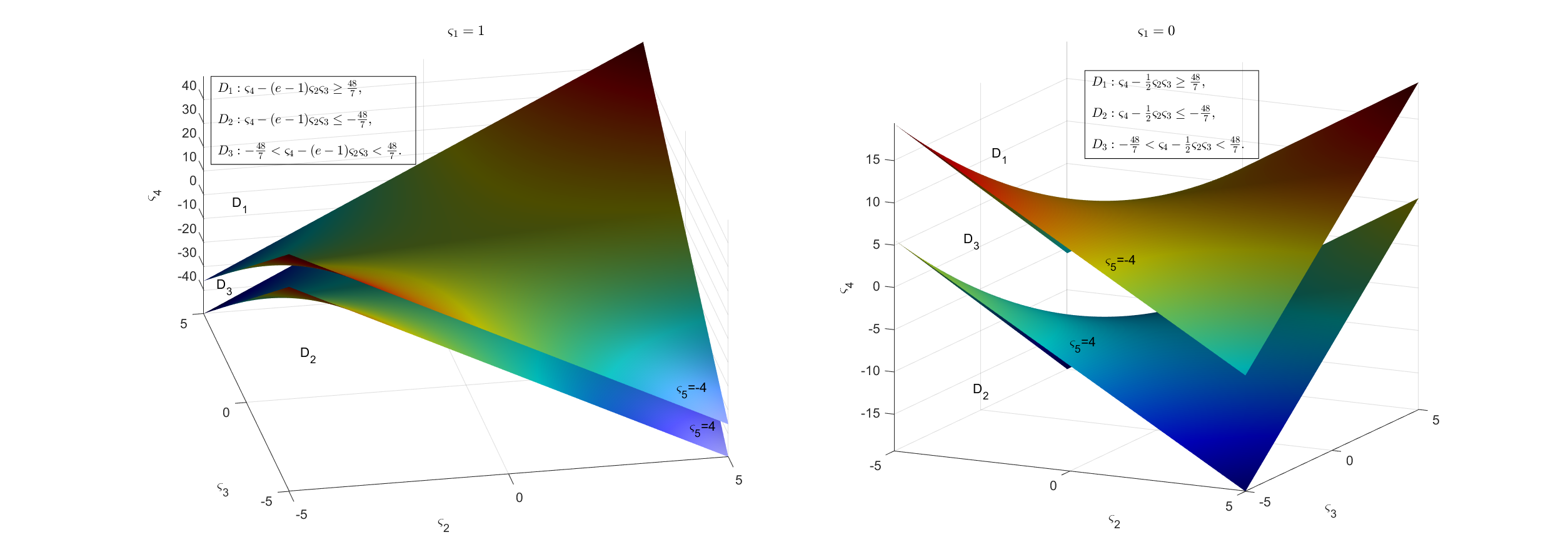}
  \caption{$[a,b]=[\frac{1}{2},1]$, $T=1$}\label{Fig.9}
\end{figure}
\end{remark}

\section{Piecewise Constant Density}
In this section, we consider the piecewise constant density with form $\widetilde{u}(\cdot)=\widetilde{u}_1\mathbb{I}_{[T_0,T_1]}(\cdot)+\sum_{m=2}^{l}\widetilde{u}_m\mathbb{I}_{(T_{m-1},T_{m}]}(\cdot),$ where $\widetilde{u}_m=(\widetilde{u}_{1m},\cdots,\widetilde{u}_{nm})^{\top}\in U^n$ ($m\in L$). For simplicity, denote $\widetilde{u}(\cdot)=\sum_{m\in L}\widetilde{u}_m\mathbb{I}_{(T_{m-1},T_{m}]}(\cdot)$. For coefficients in \eqref{1} and \eqref{9}, we assume that $\Pi\in\{A,f,\xi,\alpha,E,F,G\}$, $B_{ij}(\cdot)=\sum_{m=1}^{l}\widetilde{B}_{ijm}\mathbb{I}_{(T_{m-1},T_{m}]}(\cdot)$, and $\Pi_{i}(\cdot)=\sum_{m=1}^{l}\widetilde{\Pi}_{im}\mathbb{I}_{(T_{m-1},T_{m}]}(\cdot)$  for $i,j\in N$.  Here $\{\widetilde{A}_{im},\widetilde{B}_{ijm},\widetilde{E}_{im},\widetilde{F}_{im}\}_{m\in L}\subseteq C([T_{m-1},T_{m}],\mathcal{L}(H))$, $\{\widetilde{\xi}_{im},\widetilde{f}_{im},\widetilde{\alpha}_{im}\}_{m\in L}\subseteq H$ and $\{\widetilde{G}_{im}\}_{m\in L}\subseteq \mathcal{L}(H)$. Set $\widehat{\widetilde{V}}_m=(\widehat{\widetilde{V}}_{1m},\cdots,\widehat{\widetilde{V}}_{nm})^{\top}$ and
$$
\widehat{\widetilde{V}}_{im}=\widetilde{\alpha}_{im}+\int_{T_{m-1}}^{T_m}(\widetilde{E}_{im}(t)\widehat{\widetilde{X}}_{im}(t)+\widetilde{F}_{im}(t)\widehat{\widetilde{u}}_{im})dt+\widetilde{G}_{im}\widehat{\widetilde{X}}_{im}(T_i) \quad (i\in N,\;m\in L),
$$
where $(\widehat{\widetilde{X}}_{im},\widehat{\widetilde{u}}_{im})$ is the optimal pair satisfying \eqref{57} on time interval $(T_{m-1},T_{m}]$. Clearly, $\widehat{\widetilde{V}}_m\in H^n$ for any given $\widehat{\widetilde{u}}_m\in U^n$, and so  we can write $\widehat{\widetilde{V}}_{im}=\widetilde{J}_{im}(\widehat{\widetilde{u}}_{im})$ with $\widetilde{J}_{im}:H\rightarrow H$, i.e., the strategy $\widehat{\widetilde{u}}_{im}$ in the $i$-th market should be chosen based on the strategies $\widehat{\widetilde{u}}_{jm}$ from other $n-1$ markets.. And the WAP of the $i$-th market on the whole time interval $[s,T]$ can be given by $\sum_{m\in L}\langle\widetilde{u}_{im}, \widetilde{V}_{im}\rangle_{H}$. Set $\widehat{\widetilde{V}}_{im}=\widetilde{J}_{im}(\widehat{\widetilde{u}}_{im})$ for given $\widehat{\widetilde{u}}_{im}\in U^n$. Then, ID-NDG  and ID-MDG  can be specified as follows, respectively.

\begin{prob}\label{82}
Find the NNE $\widehat{\widetilde{u}}(\cdot)=\sum_{m\in L}\widehat{\widetilde{u}}_m\mathbb{I}_{(T_{m-1},T_{m}]}(\cdot)$  such  that
\begin{enumerate}[($\romannumeral1$)]
  \item For any $m\in L$, one has $\widehat{\widetilde{u}}_m=(\widehat{\widetilde{u}}_{1m},\cdots,\widehat{\widetilde{u}}_{nm})^{\top}\in U^n$;
  \item For any $i\in N$, any $m\in L$ and a.e. $x\in[a,b]$, one has
$\widehat{\widetilde{u}}_{im}(x)>0\Rightarrow\widehat{\widetilde{V}}_{im}(x)
\geq
\langle\widehat{\widetilde{V}}_{im},\widehat{\widetilde{u}}_{im}
\rangle_{H}.$
\end{enumerate}
\end{prob}

\begin{prob}\label{83}
Find the MNE   $\widehat{\widetilde{u}}(\cdot)=\sum_{m\in L}\widehat{\widetilde{u}}_m\mathbb{I}_{(T_{m-1},T_{m}]}(\cdot)$  such  that
\begin{enumerate}[($\romannumeral1$)]
  \item For any $m\in L$, one has $\widehat{\widetilde{u}}_m=(\widehat{\widetilde{u}}_{1m},\cdots,\widehat{\widetilde{u}}_{nm})^{\top}\in U^n$;
\item  For any $i\in N$, one has
$\sum_{m\in L}\langle \widehat{\widetilde{u}}_{im},\widehat{\widetilde{V}}_{im}\rangle_H=\max_{\widetilde{u}(\cdot)}[\sum_{m\in L}\langle\widetilde{u}_{im}, \widetilde{V}_{im}\rangle_{H}]$,
where $\widetilde{V}_{im}=\widetilde{J}_{im}(\widetilde{u}_{im})$.
\end{enumerate}
\end{prob}

Under mild conditions, the existence, uniqueness and stability of both the NNE of Problem \ref{82} and the MNE of Problem \ref{83} can be guaranteed because of the following reasons:
(i) Problem \ref{82}: for  time interval $[T_{m-1},T_{m}]$ with any given $m\in L$, the existence, uniqueness and stability of the density $\widehat{\widetilde{u}}_m$ can be obtained by similar arguments as in Section 3; (ii) Problem \ref{83}: for any given $i\in N$,   finding $\widehat{\widetilde{u}}(\cdot)$ to maximize the WAP  is equivalent to finding $\widehat{\widetilde{u}}_m$ to maximize $\langle\widetilde{u}_{im}, \widetilde{V}_{im}\rangle_{H}$ on time interval $[T_{m-1},T_{m}]$ for each $m\in L$, while each $\widehat{\widetilde{u}}_m$ can be solved by the results obtained in Section 4.

As a consequence, the results obtained in Sections 3-4 can be generalized to the case of piecewise  constant density.

\section{Conclusions}
In this paper, we propose the novel models of the ID-NDG and the ID-MDG. Concerning on constant density, with the help of variational analysis, we obtain the unique NNE of the ID-NDG and establish the stability of the NNE under positive definite condition. Complementarily,  we  obtain the unique MNE of the ID-MDG and establish the stability of the MNE under negative definite condition. The two  conditions with certain symmetry are crucial for strategic decision making regarding competitive versus cooperative market behaviors. Finally, the aforementioned results are generalized to the case of piecewise  constant density.

It is worth emphasizing that several promising extensions of our problem formulation emerge as significant directions for future research: (i) As highlighted in Remark \ref{77}, the case when $\varsigma_5=0$ requires special consideration in equilibrium analysis;  (ii) As pointed out in Remark \ref{80}, it would be interesting to consider the piecewise constant density in future works; (iii) Practical applications often require incorporating demand constraints, motivating the study of generalized affine constraints:
$$
\int_s^T\langle X_i(t),\varsigma_8(t)\rangle_Hdt+\langle u_i,\varsigma_9\rangle_H+\langle X_i(T),\varsigma_{10}\rangle_H\leq\varsigma_{11}, \quad\forall i\in N,
$$
where $\varsigma_8\in C([s,T],H)$, $\varsigma_9,\varsigma_{10}\in H$, $\varsigma_{11}\in \mathbb{R}$ are all given.  This extension would significantly expand the model's applicability to real-world resource allocation problems with constraints on  the variables. For instance,   total demand $\int_s^T\langle X_i(t),1\rangle_Hdt$,   total production quantities $(T-s)\int_a^b xu_i(x)dx$, and  terminal demand $\langle X_i(T),1\rangle_H$ should be confined within specific bounds to ensure operational feasibility and economic viability.

These research directions present theoretically rich and practically relevant extensions of our current framework. We intend to systematically address these challenges through continued investigation, developing corresponding analytical tools and computational methods to handle the increased mathematical complexity inherent in each extension.

\bibliographystyle{apalike}
\bibliography{References}

\end{document}